\documentclass[11pt,reqno]{amsart}
\usepackage{amsmath,amssymb,mathrsfs,amscd}
\usepackage{graphicx,subfigure,cite}
\usepackage{color}
\usepackage{array}
\usepackage [section]{placeins}
\newtheorem{theorem}{Theorem}[section]

\newtheorem{lemma}[theorem]{Lemma}

\numberwithin{equation}{section}

\allowdisplaybreaks[4]
\title[Linear FEM for flexural wave scattering]{Numerical solution of the cavity scattering problem for flexural waves on thin plates: linear finite element methods}

\author{Junhong Yue}
\address{College of Computer Science and Technology (College of Data Science), Taiyuan University of Technology, Shanxi 030024, China}
\email{yuejunhong@tyut.edu.cn}

\author{Peijun Li}
\address{Department of Mathematics, Purdue University, West Lafayette, Indiana 47907, USA}
\email{lipeijun@math.purdue.edu}

\thanks{The first author is supported by the NSFC grant 11901423. The second author is supported in part by the NSF grant DMS-2208256.}



\date{}

\begin{document}
\maketitle

\begin{abstract}
Flexural wave scattering plays a crucial role in optimizing and designing structures for various engineering applications. Mathematically, the flexural wave scattering problem on an infinite thin plate is described by a fourth-order plate-wave equation on an unbounded domain, making it challenging to solve directly using the regular linear finite element method (FEM). In this paper, we propose two numerical methods, the interior penalty FEM (IP-FEM) and the boundary penalty FEM (BP-FEM) with a transparent boundary condition (TBC), to study flexural wave scattering by an arbitrary-shaped cavity on an infinite thin plate. Both methods decompose the fourth-order plate-wave equation into the Helmholtz and modified Helmholtz equations with coupled conditions at the cavity boundary. A TBC is then constructed based on the analytical solutions of the Helmholtz and modified Helmholtz equations in the exterior domain, effectively truncating the unbounded domain into a bounded one. Using linear triangular elements, the IP-FEM and BP-FEM successfully suppress the oscillation of the bending moment of the solution at the cavity boundary, demonstrating superior stability and accuracy compared to the regular linear FEM when applied to this problem.
\end{abstract}

\section{Introduction}

Flexural wave scattering is a widespread phenomenon with practical engineering applications, including the design of lightweight mechanical structures with low noise \cite{APM-16}, ultra-broadband elastic cloaking devices \cite{DZA-18, FGE-09}, Platonic diffraction gratings and arrays \cite{H-14, HCM-16}, massive floating concrete runways offshore, and health monitoring of thin-walled structures such as aircraft wings, oil tank walls, and pressure vessels \cite{V-07}. This phenomenon arises from the interaction of incident waves with different types of scatterers (e.g., voids, rigid, and elastic scatterers) on thin-wall structures \cite{LHH-03}, attracting considerable attention in the research community. For instance, Akrucci et al. \cite{APM-16} investigated the effect of acoustic black holes on flexural wave scattering on infinite thin plates, effectively reducing plate oscillations without increasing structure mass. Liu et al. \cite{LZ-19} realized a broadband cylindrical cloak for flexural waves in elastic thin plates using nonlinear transformation, guiding flexural waves more effectively outside the cloak region. Haslinger et al. \cite{HMM-17} studied scattering and transmission of flexural waves in a thin plate with a semi-infinite array of point scatterers, demonstrating dynamically anisotropic wave effects in semi-infinite platonic crystals. Evans et al. \cite{EP-07} explored flexural wave scattering in an elastic thin plate floating on water. Wang et al. \cite{WC-05} investigated the scattering behavior of extensional and flexural plate waves by a cylindrical inhomogeneity for structural health monitoring, characterizing the interaction of plate waves with structural damage. Consequently, studying flexural wave scattering in thin-walled structures holds great importance for optimizing and designing new structures.

Mathematically, the scattering problem of flexural waves on infinite thin plates can be described by a fourth-order plate-wave equation on an unbounded domain. Analytical solutions for such problems are only attainable for isotropic thin plates containing scatterers with simple geometries, posing challenges for complex geometries and media. Apparently, numerical methods are necessary to simulate and solve these problems. Several numerical approaches have been proposed for solving flexural wave scattering problems. For instance, Norris et al. \cite{NV-95} explored the energy flux conservation and the optical theorem in the context of flexural wave scattering, applying them to flexural scattering by circular scatterers. Matus et al. \cite{ME-10} employed the transfer matrix (T-matrix) method for flexural wave scattering by a single noncircular scatterer. Climente et al. \cite{CGW-17, CNS-15} utilized T-matrix and impedance matrix methods to study flexural wave scattering by a hole containing beam resonators in an infinite thin plate for vibration control. Lee et al. \cite{LC-10-1, LC-10-2, LC-11} investigated flexural wave scattering in a thin plate with multiple circular inclusions using the multipole method, the multipole Trefftz method, and the null-field integral equation approach. Cai et al. \cite{CH-16} applied the T-matrix method for the multiple scattering of flexural waves by different types of circular scatterers on thin plates. Wang et al. \cite{WLD-19} studied the multiple scattering of flexural waves by varying-thickness annular inclusions on infinite thin plates using a semi-analytical method. Wang et al. \cite{WL-21} developed a semi-analytical model for a novel plate-harvester system, combining multiple flexural scattering theories of thin plates with coupled electroelastic dynamics of piezoelectric composite beams. Dong et al. \cite{DL-23} proposed a novel formulation of boundary integral equations for the scattering of flexural waves by obstacles on infinite thin plates.

However, the finite element method (FEM), known for its efficiency and stability in numerical algorithms, is rarely employed to solve the scattering problem of flexural waves due to the challenges posed by the unbounded domain of the problem and the presence of high-order partial differential equations (PDEs). To address the unbounded domain issue, common truncated techniques, such as absorbing boundary condition (ABC) \cite{CK-83, EM-77}, perfectly matched layer (PML) \cite{B-94}, and transparent boundary condition (TBC) \cite{JLZ-13, JLLZ-17, LY-22}, are used to truncate the unbounded domain into bounded computational domains when applying FEM. However, these techniques are primarily constructed for second-order PDEs in exterior domains, such as acoustic, electromagnetic, and elastic wave scattering problems. Directly applying these techniques to flexural wave scattering problems, which involve a fourth-order wave equation, is challenging. Therefore, the development of an effective truncation technique specifically for flexural wave scattering is crucial. Moreover, to solve the fourth-order wave equation within the truncation domain, various numerical methods have been proposed, including meshless methods \cite{TFC-20} and different types of FEMs. In the literature, classical FEM approaches include $C^1$ conforming FEMs like the Argyris element \cite{C-02} with 21 degrees of freedom, as well as nonconforming FEMs like the Adini element \cite{HS-13} and Morley element \cite{WX-06}. However, these elements are rarely practical due to either their high number of degrees of freedom or the complexity involved in their implementation. It remains a challenge to find a simpler and more efficient approach to solve the fourth-order wave equation in the truncation domain.

As a result, numerous FEMs based on linear triangular elements have been proposed to solve fourth-order problems, including mixed methods \cite{AD-01, BG-11, BF-12, GP-78, GNP-08} and the recovery-based linear FEM \cite{CZZ-16, GZZ-18, HWY-20, L-14}. While mixed methods are widely used, they require careful treatment of essential and natural boundary conditions. In particular, the Ciarlet--Raviart mixed FEM exhibits oscillation behavior of the middle variable on the boundary, necessitating the addition of corresponding penalty terms \cite{AD-01}. On the other hand, the recovery-based linear FEM is a nonconforming FEM that discretizes the Laplace operator by utilizing the gradient recovery operator acting on the gradient of the $C^0$ linear element.

This paper first reduces the biharmonic plate wave equation to the Helmholtz and modified Helmholtz equations with coupled conditions at the cavity boundary using two auxiliary functions. To truncate the unbounded domain into a bounded one, the transparent boundary conditions (TBCs) are constructed using Fourier series solutions of the Helmholtz and modified Helmholtz equations, satisfying the Sommerfeld radiation conditions. The finite element approximation with linear triangular elements \cite{CMS-10} is then utilized to solve the decomposed problem with coupled boundary conditions. However, the solutions obtained by this method exhibit oscillatory behavior at the cavity boundary. Motivated by recent works \cite{AD-01, DW-15, LZZ-20, W-14}, we introduce the interior penalty term or the boundary penalty term to the variational formulations of the Helmholtz and modified Helmholtz equations, effectively suppressing the oscillation of the bending moment of the solution at the cavity boundary.

This paper focuses on the numerical computation of flexural wave scattering by an arbitrary shaped cavity with the clamped boundary and contributes in the following four aspects:

\begin{enumerate}

 \item Construction of TBCs for flexural wave cavity scattering in two dimensions, which is equivalent to satisfying the Sommerfeld radiation conditions.

 \item Deduction of a decomposed problem of the biharmonic plate wave equation by introducing two auxiliary functions, and proof of the uniqueness of its solution.

 \item Proposal of the linear finite element method with interior penalty term (IP-FEM) and boundary penalty term (BP-FEM) for solving the coupled boundary Helmholtz and modified Helmholtz equations, providing stable numerical solutions.

 \item Construction of an analytical solution for flexural wave scattering by a circular cavity with a clamped boundary, facilitating comparative analysis.

\end{enumerate}

In this work, we propose the IP-FEM and the BP-FEM as numerical methods to simulate cavity scattering in an infinite thin plate. The paper is outlined as follows. Section 2 describes the cavity scattering problems using the biharmonic plate wave equation with the clamped boundary condition. In Section 3, we construct the transparent boundary conditions (TBCs) to truncate the unbounded domain for plate wave scattering. Section 4 presents a decomposed problem by introducing two auxiliary functions for the plate wave equation, along with the proof of the uniqueness of this decomposition problem. In Section 5, we provide the variational formulations and discretized systems of the IP-FEM and the BP-FEM. Section 6 showcases numerical experiments to validate the effectiveness of the proposed methods, comparing them with analytical solutions or reference solutions. Finally, in Section 7, we draw conclusions from this study.

\section{Problem formulation}

Let us consider a cavity denoted as $D$, located within an infinite thin plate, with a Lipschitz continuous boundary $\partial D$. The structure is illuminated by a time-harmonic plane wave represented by $u^{\rm inc}(\boldsymbol x)=e^{{\rm i}\kappa \boldsymbol{x}\cdot \boldsymbol{d}}$, where $\kappa >0$ is the wavenumber determined by $\kappa^4 = \omega^2\rho h/D_c$. Here, $\omega$ denotes the angular frequency, while $\rho$, $h$, and $D_c$ refer to the mass density, thickness, and flexural rigidity of the plate, respectively. The incident direction is given by $\boldsymbol{d} = (\cos(\alpha),\sin(\alpha))$, with $\alpha \in [0,2\pi)$ representing the incident angle.

It can be verified that the incident wave field $u^{\rm inc}$ satisfies the two-dimensional biharmonic wave equation given by:
\begin{equation}\label{GE_Incident}
\Delta^2 u^{\rm inc} - \kappa^4 u^{\rm inc} = 0\quad \text{in} ~\mathbb R^2.
\end{equation}
The out-of-plane displacement of the plate, denoted as $u$, satisfies the two-dimensional biharmonic wave equation in the exterior of $D$, which is expressed as:
\begin{equation} \label{GE_Total}
\Delta^2 u - \kappa^4 u = 0 \quad \text{in} ~ \mathbb{R}^2\setminus \overline{D}.
\end{equation}
The total field $u$ is assumed to satisfy the following clamped boundary condition on $\partial D$:
\begin{equation} \label{BC_Total}
u = 0,\quad \partial_\nu u =0,
\end{equation}
where $\nu$ is the unit normal vector on $\partial D$.

It can be observed from \eqref{GE_Incident}--\eqref{BC_Total} that the scattered field $v = u-u^{\rm inc}$ satisfies:
\begin{equation} \label{GE_scattered}
\Delta^2 v - \kappa^4 v = 0 \quad \text{in} ~ \mathbb{R}^2\setminus \overline{D},
\end{equation}
and the following boundary conditions on $\partial D$:
\begin{equation} \label{BC1_scattered}
v = -u^{\rm inc},\quad \partial_\nu v = -\partial_\nu u^{\rm inc}.
\end{equation}
In addition, the scattered field $v$ and its Laplacian $\Delta v$ are required to satisfy the Sommerfeld radiation condition:
\begin{equation} \label{BC2_scattered}
\lim_{r\rightarrow \infty}r^{1/2}(\partial_r v-{\rm i}\kappa v) = 0,\quad
\lim_{r\rightarrow \infty}r^{1/2}(\partial_r \Delta v-{\rm i}\kappa \Delta v) = 0, \quad r = |x|.
\end{equation}

We introduce standard notations used in this paper. Let $B_R=\{\boldsymbol x\in\mathbb R^2: |\boldsymbol x|<R\}$ represent a disc with boundary $\Gamma_R=\{\boldsymbol x\in\mathbb R^2: |\boldsymbol x|=R\}$. The radius $R$ is chosen to be sufficiently large such that the cavity $D$ is completely contained within $B_R$. We denote $\Omega=B_R\setminus\overline D$. The inner product and the norm in $L^2(\Omega)$ are
\[
(\phi,\psi) = \int_{\Omega}\phi\bar\psi d\boldsymbol x,\quad \|\phi\|_{0,\Omega} = (\phi,\phi)^{1/2}.
\]
Define $H_{\partial D}^1 = \{\phi \in H^1(\Omega):\phi=0 ~\text{on} ~ \partial D\}$,
$H_{\Omega}^1 = \{\phi \in H^1(\Omega):\phi=0 ~\text{on} ~ \overline{\Omega}\backslash \partial D\}$, and
$H^1 = H_{\partial D}^1 \cup H_{\Omega}^1$. It is known that $H^{-1/2}(\Gamma_R)$ is the dual space of $H^{1/2}(\Gamma_R)$ with respect to the inner product
\[
\langle \phi,\psi\rangle_{\Gamma_R} = \int_{\Gamma_R}\phi\bar\psi ds.
\]

\section{Transparent boundary conditions}\label{sec3}

In this section, we introduce the transparent boundary condition (TBC) on $\Gamma_R$ to transform the cavity scattering problem from the open domain $\mathbb{R}^2\setminus\overline{D}$ to the bounded domain $\Omega$. This allows us to truncate the unbounded domain into a bounded one for numerical simulation and analysis.

Following \cite{DL-23}, we consider two auxiliary functions $v_H$ and $v_M$, defined as:
\begin{equation*}
 v_H=-\frac{1}{2\kappa^2}(\Delta v-\kappa^2 v),\quad v_M=\frac{1}{2\kappa^2}(\Delta v+\kappa^2 v).
\end{equation*}
It can be observed that the quantities $v$, $\Delta v$, $v_H$, and $v_M$ are related through the equations
\begin{equation}\label{Rel_HM_1}
v = v_H + v_M,\quad
\Delta v = \kappa^2 (-v_H+v_M),
\end{equation}
and
\begin{equation}\label{Rel_HM_2}
v_H = \frac{1}{2}(v - \kappa^{-2} \Delta v),\quad
v_M = \frac{1}{2}(v + \kappa^{-2} \Delta v).
\end{equation}

The biharmonic wave equation \eqref{GE_scattered} can be written as
\[
 (\Delta^2-\kappa^4) v = (\Delta-\kappa^2)(\Delta+\kappa^2)(v_H + v_M) = 0\quad\text{in} ~\mathbb R^2\setminus\overline D,
\]
which implies that $v_H$ and $v_M$ satisfy the Helmholtz equation and the modified Helmholtz equation, respectively:
\begin{equation} \label{HM_eq}
\Delta v_H + \kappa^2 v_H = 0,\quad \Delta v_M - \kappa^2 v_M = 0.
\end{equation}
Combining \eqref{BC2_scattered} and \eqref{HM_eq}, we deduce that the functions $v_H$ and $v_M$ satisfy the Sommerfeld radiation condition:
\begin{equation}\label{RC}
 \lim_{r\rightarrow \infty}r^{1/2}(\partial_r v_H-{\rm i}\kappa v_H) = 0,\quad
\lim_{r\rightarrow \infty}r^{1/2}(\partial_r v_M-{\rm i}\kappa v_M) = 0.
\end{equation}

From \eqref{HM_eq} and \eqref{RC}, we can conclude that $v_H$ and $v_M$ have the following Fourier series expansions in $\mathbb{R}^2\setminus\overline{B_R}$:
\begin{equation*}
v_H(r,\theta) = \sum_{n \in \mathbb{Z}} \frac{H_n^{(1)}(\kappa r)}{H_n^{(1)}(\kappa R)}v_H^{(n)}(R)e^{{\rm i} n\theta}, \quad
v_M(r,\theta) = \sum_{n \in \mathbb{Z}} \frac{K_n(\kappa r)}{K_n(\kappa R)}v_M^{(n)}(R)e^{{\rm i} n\theta},
\end{equation*}
where $H_n^{(1)}$ is the Hankel function of the first kind with order $n$, $K_n$ is the modified Bessel function with order $n$, $v_H^{(n)}$ and $v_M^{(n)}$ are the Fourier coefficients given by
\[
 v_H^{(n)}(R) = \frac{1}{2\pi}\int_0^{2\pi} v_H(R,\theta) e^{-{\rm i}n\theta} d\theta,\quad  v_M^{(n)}(R) = \frac{1}{2\pi}\int_0^{2\pi} v_M(R,\theta) e^{-{\rm i}n\theta} d\theta.
\]

For any function $\phi\in L^2(\Gamma_R)$ with the Fourier series expansion
\[
\phi(R,\theta) = \sum_{n \in \mathbb{Z}}\phi^{(n)}(R)e^{{\rm i}n \theta},\quad
\phi^{(n)}(R) = \frac{1}{2\pi}\int_0^{2\pi} \phi(R,\theta) e^{-{\rm i}n\theta} d\theta,
\]
we define two Dirichlet-to-Neumann (DtN) operators
\begin{align}\label{TBC_HM}
T_1 \phi = \frac{1}{R}\sum_{n \in \mathbb{Z}}h_n(\kappa R) \phi^{(n)}(R)e^{{\rm i} n\theta},\quad
T_2 \phi = \frac{1}{R}\sum_{n \in \mathbb{Z}}k_n(\kappa R) \phi^{(n)}(R)e^{{\rm i} n\theta},
\end{align}
where
\[
h_n(z) = z \frac{H_n^{(1)'}(z)}{H_n^{(1)}(z)}, ~~~ k_n(z) = z \frac{K'_n(z)}{K_n(z)}.
\]
Using \eqref{TBC_HM}, we deduce the TBC on $\Gamma_R$:
\begin{equation}\label{NormalDer_HM}
\partial_r v_H = T_1 v_H,\quad  \partial_r v_M = T_2 v_M.
\end{equation}

Given that $v$ and $\Delta v$ are periodic functions with respect to $\theta$ on $\Gamma_R$, they can be represented by their Fourier series expansions:
\begin{align*}
v(R,\theta) = \sum_{n \in \mathbb{Z}}v^{(n)}(R)e^{{\rm i} n\theta},\quad \Delta v(R,\theta) = \sum_{n \in \mathbb{Z}}(\Delta v)^{(n)}(R)e^{{\rm i} n\theta}.
\end{align*}
By combining \eqref{NormalDer_HM} and \eqref{Rel_HM_1}--\eqref{Rel_HM_2}, we obtain the TBC for the scattered fields:
\begin{align*}
\partial_r v &= T_1 v_H + T_2 v_M\\
        &= \frac{1}{2}T_1(v-\kappa^{-2} \Delta v) + \frac{1}{2}T_2(v+\kappa^{-2}\Delta v)\\
        &= \frac{1}{2}(T_1+T_2)v + \frac{1}{2\kappa^2}(T_2-T_1)\Delta v
\end{align*}
and
\begin{align*}
\partial_r \Delta v &= \kappa^2 (-T_1 v_H + T_2 v_M)\\
        &= \kappa^2 (-\frac{1}{2}T_1(v-\kappa^{-2}\Delta v)
          + \frac{1}{2}T_2(v+\kappa^{-2}\Delta v))\\
        &= \frac{\kappa^2}{2}(T_2-T_1)v + \frac{1}{2}(T_1+T_2)\Delta v.
\end{align*}
Using the above equations, we deduce the TBC for the total fields:
\begin{equation}\label{TBC_Total_1}
\partial_r u = \frac{1}{2} (T_1+T_2)u + \frac{1}{2\kappa^2}(T_2-T_1)\Delta u+ g_1
\end{equation}
and
\begin{equation}\label{TBC_Total_2}
\partial_r \Delta u = \frac{\kappa^2}{2}(T_2-T_1)u + \frac{1}{2}(T_1+T_2)\Delta u + g_2,
\end{equation}
where
\[
g_1 =\partial_r u^{\rm inc} -T_1 u^{\rm inc}, \quad g_2=-\kappa^2 g_1.
\]

\section{The decomposed problem}\label{sec4}

In this section, we present a decomposed formulation for the cavity scattering problem, consisting of \eqref{GE_Total}--\eqref{BC_Total} and \eqref{TBC_Total_1}--\eqref{TBC_Total_2}.

Let us consider two auxiliary functions defined as follows:
\begin{align} \label{Aux_Fun1}
p=\frac{1}{2\kappa^2}(\Delta u -\kappa^2 u),\quad q = \frac{1}{2\kappa^2}(\Delta u + \kappa^2 u).
\end{align}
It is evident that $p$ satisfies the Helmholtz equation, while $q$ satisfies the modified Helmholtz equation. Moreover, it can be verified that
\begin{align} \label{Relate_pq_uw}
u = q-p,\quad \Delta u = \kappa^2(p+q),
\end{align}
and
\begin{align} \label{Aux_Fun2}
p = -u_H,\quad q = u_M,
\end{align}
where $u_H$ and $u_M$ are the Helmholtz and modified Helmholtz wave components of $u$, respectively.

Using \eqref{Aux_Fun1}--\eqref{Aux_Fun2}, we obtain the following boundary value problem for $p$ and $q$:
\begin{equation} \label{PQ_bvp_Total}
\left\{
\begin{aligned}
& \Delta p + \kappa^2 p = 0,\quad  \Delta q - \kappa^2 q = 0 & \text{in} ~ \Omega,\\
& p-q = 0,\quad \partial_\nu p -\partial_\nu q = 0 & \text{on} ~ \partial D,\\
& \partial_r p = T_1 p -g_1,\quad \partial_r q = T_2 q & \text{on} ~ \Gamma_R.
\end{aligned}
\right.
\end{equation}

Equivalently, we may consider two auxiliary functions for the scattered field:
\[
p^s = \frac{1}{2\kappa^2}(\Delta v -\kappa^2 v), \quad q^s = \frac{1}{2\kappa^2}(\Delta v + \kappa^2 v).
\]
Hence, we have
\[
v = q^s-p^s, \quad \Delta v = \kappa^2(p^s+q^s),
\]
and
\[
p^s = -v_H, \quad q^s = v_M.
\]
It can be verified that $p^s$ and $q^s$ satisfy the following boundary value problem:
\begin{equation} \label{PQ_bvp_Scattered}
\left\{
\begin{aligned}
& \Delta p^s + \kappa^2 p^s = 0,\quad  \Delta q^s - \kappa^2 q^s = 0 & \text{in} ~ \Omega,\\
& p^s-q^s = u^{\rm inc},\quad \partial_\nu p^s -\partial_\nu q^s = \partial_\nu u^{\rm inc} & \text{on} ~ \partial D,\\
& \partial_r p^s = T_1 p^s,\quad \partial_r q^s = T_2 q^s & \text{on} ~ \Gamma_R.
\end{aligned}
\right.
\end{equation}

\begin{lemma}\label{Lemma1}
Let $z$ be a positive real number. Then
\begin{align*}
\Re (h_n(z))<0,\quad \Im (h_n(z)) >0, \quad
\Re(k_n(z))< 0,\quad  \Im (k_n(z))=0.
\end{align*}
\end{lemma}

\begin{proof}
Using the definition $H_n^{(1)}(z)= J_n(z) + {\rm i} Y_n(z)$, we can express $h_n(z)$ into
\begin{align*}
h_n(z) = \frac{z}{|H_n^{(1)}(z)|^2}\left(J_n^{\prime}(z)+{\rm i}Y_n^{\prime}(z)\right)\left(J_n(z)-{\rm i}Y_n(z)\right),
\end{align*}
where the real-valued functions $J_n(z)$ and $Y_n(z)$ are the Bessel functions of the first kind and second kind with order $n$, respectively.

First, we consider the real part of $h_n(z)$:
\begin{align*}
\Re(h_n(z)) & = \frac{z}{|H_n^{(1)}(z)|^2}\left(J_n^{\prime}(z)J_n(z)+Y_n^{\prime}(z)Y_n(z)\right)\\
& =\frac{z}{2|H_n^{(1)}(z)|^2}\frac{d}{dz}\left(J_n(z)^2+Y_n(z)^2\right).
\end{align*}
By the Nicholson's integral \cite[$(10.9.30)$]{OLBC-10},
\[
J_n(z)^2+Y_n(z)^2=\frac{8}{\pi^2} \int_0^{\infty}\cosh(2nt)K_0(2z\sinh(t))dt,
\]
where $\sinh(t)$ and $\cosh(t)$ are the hyperbolic sine and hyperbolic cosine functions, respectively, we have
\[
\frac{d}{dz}\left(J_n(z)^2+Y_n(z)^2\right)=\frac{16}{\pi^2} \int_0^{\infty}\cosh(2nt)\sinh(t)K_0^{\prime}(2z\sinh(t))dt.
\]
Given that $\cosh(t)>0$ for $t \in \mathbb{R}$ and $\sinh(t)>0$ for $t>0$, along with the fact that $K_n(z)$ is positive and decreasing throughout the interval $0<z<\infty$ for $n\geq 0$ (cf. \cite[Section 10.37]{OLBC-10}), and $K_0^{\prime}(z) = -K_1(z)<0$, we can conclude that $\Re(h_n(z))<0$.

Second, we consider the imaginary part of $h_n(z)$:
\begin{align*}
\Im(h_n(z)) = \frac{z}{|H_n^{(1)}(z)|^2}\left(Y_n^{\prime}(z)J_n(z)-J_n^{\prime}(z)Y_n(z)\right).
\end{align*}
Using the identities
\[
2J_n^{\prime} = J_{n-1}(z)-J_{n+1}(z), \quad 2Y_n^{\prime}(z) = Y_{n-1}(z)-Y_{n+1}(z),
\]
and the Wronskian \cite{OLBC-10}:
\[
W\left\{J_n(z),Y_n(z)\right\}=J_{n+1}(z)Y_n(z)-J_n(z)Y_{n+1}(z)=2/(\pi z),
\]
we have
\begin{align*}
\Im(h_n(z))& = \frac{z}{|H_n^{(1)}(z)|^2}\left((Y_{n-1}(z)-Y_{n+1}(z))J_n(z)-(J_{n-1}(z)-J_{n+1}(z))Y_n(z)\right).\\
& = \frac{2}{\pi}\frac{1}{|H_n^{(1)}(z)|^2}>0.
\end{align*}

Next, we examine the properties of $\Re(k_n(z))$ and $\Im(k_n(z))$. Since $K_n(z)$ is a real-valued function, $K_n^{\prime}(z)$ is also real, implying that $\Im(k_n(z))=0$. For a given $n$ and $z>0$, $K_n(z)>0$, and $K_n(z)$ is monotonically decreasing with respect to $z$, i.e., $K'_n(z)<0$. Consequently, we have $\Re(k_n(z))<0$ for $z>0$.
\end{proof}

\begin{theorem}\label{Th1}
The coupled boundary value problem \eqref{PQ_bvp_Total} has at most one solution for $\kappa>0$.
\end{theorem}

\begin{proof}
It suffices to show that $p = 0$ and $q = 0$ in $\Omega$ when $g_1 = 0$.
Applying Green's theorem in $\Omega$ and the boundary condition, we obtain
\begin{align*}
& (\nabla p,\nabla p) - \kappa^2(p,p) -\langle T_1 p,p\rangle_{\Gamma_R}
 -\langle \partial_\nu p,p\rangle_{\partial D} = 0,\\
& (\nabla q,\nabla q) + \kappa^2(q,q) -\langle T_2 q,q\rangle_{\Gamma_R}
 -\langle \partial_\nu q,q\rangle_{\partial D} = 0.
\end{align*}
Since $\langle \partial_\nu p,p\rangle_{\partial D} = \langle \partial_\nu q,q\rangle_{\partial D}$ on $\partial D$, we have
\begin{align*}
& (\nabla p,\nabla p) - \kappa^2(p,p) -\langle T_1 p,p\rangle_{\Gamma_R}
 = (\nabla q,\nabla q) + \kappa^2(q,q) -\langle T_2 q,q\rangle_{\Gamma_R}.
\end{align*}
A simple calculation yields
\begin{align*}
\langle T_1 p,p\rangle_{\Gamma_R} = 2\pi\sum_{n\in \mathbb{Z}} h_n(\kappa R) |p^{(n)}|^2,\quad
\langle T_2 q,q\rangle_{\Gamma_R} = 2\pi\sum_{n\in \mathbb{Z}} k_n(\kappa R) |q^{(n)}|^2,
\end{align*}
where $p^{(n)}$ and $q^{(n)}$ are the Fourier coefficients of $q$ and $q$ on $\Gamma_R$.
Taking the imaginary part of the above equation gives
\begin{align*}
\Im \{-\langle T_1 p,p\rangle_{\Gamma_R} + \langle T_2 q,q\rangle_{\Gamma_R}\}
& = -2\pi\sum_{n\in \mathbb{Z}} \Im (h_n(\kappa R)) |p^{(n)}|^2 \\
&\quad + 2\pi\sum_{n\in \mathbb{Z}} \Im (k_n(\kappa R)) |q^{(n)}|^2 = 0,
\end{align*}
which gives $p^{(n)}=0$ for $n\in\mathbb{Z}$ using Lemma \ref{Lemma1}. Thus we have $p=0$ and $\partial_r p = 0$ on $\Gamma_R$. According to the Holmgren uniqueness theorem, we obtain $p=0$ and $\partial_\nu p = 0$ in $\mathbb{R}^2\setminus \overline{B_R}$. Furthermore, a unique continuation result implies that $p = 0$ and $\partial_\nu p = 0$ in $\Omega$. Considering the boundary conditions on $\partial D$, we find that $q=0$ and $\partial_\nu q = 0$ on $\partial D$. Consequently, by applying the Holmgren uniqueness theorem, we can deduce that $q=0$ in $\Omega$.
\end{proof}

To solve the decomposed problem \eqref{PQ_bvp_Total} by using the linear FEM, we introduce its variational formulation. Using the test functions ${\phi,\psi,\varphi}\in H_{\partial D}^1(\Omega) \times H_{\partial D}^1(\Omega)\times H_{\Omega}^1$, the weak formulation of \eqref{PQ_bvp_Total} aims to find ${p,q} \in H^1(\Omega) \times H^1(\Omega)$ that satisfy the following equations:
\begin{equation}\label{Var1_pq}
\left\{
\begin{aligned}
& b_1(p,\phi) = -\langle g_1,\phi\rangle_{\Gamma_R} &\forall ~ \phi \in H_{\partial D}^1(\Omega),\\
& b_2(q,\psi) = 0 &\forall ~ \psi \in H_{\partial D}^1(\Omega),\\
\end{aligned}
\right.
\end{equation}
and
\begin{equation}\label{Var2_pq}
(\nabla(p-q), \nabla\varphi) - \kappa^2((p+q),\varphi) = 0 \quad \forall ~ \varphi \in H_{\Omega}^1,
\end{equation}
where $p = p_0 + p_D$ and $q = q_0 + p_D$ with $p_0, q_0 \in H_{\partial D}^1(\Omega)$ and $p_D \in H_{\Omega}^1$.
Here the sesquilinears $b_1:H^1 \times H_{\partial D}^1\rightarrow \mathbb{C}$ and $b_2:H^1 \times H_{\partial D}^1\rightarrow \mathbb{C}$ are defined by
\begin{align*}
b_1(\phi, \psi) &= (\nabla \phi,\nabla \psi) -\kappa^2(\phi, \psi)
    -\langle T_1 \phi, \psi\rangle_{\Gamma_R},\\
 b_2(\phi, \psi) & = (\nabla\phi,\nabla\psi) + \kappa^2(\phi, \psi)
    -\langle T_2 \phi, \psi\rangle_{\Gamma_R}.
\end{align*}

\section{The linear finite element methods}\label{sec5}

In this section, we introduce the IP-FEM and BP-FEM methods for solving the problem \eqref{Var1_pq}--\eqref{Var2_pq}. First, we define the linear finite element spaces and the corresponding symbols. Next, we construct the variational formulations by incorporating an interior penalty term and a boundary penalty term, respectively. Finally, we present the discretized systems using the linear FEM.

\subsection{Finite element spaces}

Let $\mathcal M_h$ be a triangulation of $\Omega$ such that $\overline{\Omega} = \cup_{K \in\mathcal M_h}K$, where $K$ denotes a triangular element. Let $\mathcal{C}_h^I$ and $\mathcal{C}_h^B$ be the set of all interior and boundary edges of mesh $\mathcal M_h$, respectively.

We define the finite element space using piecewise linear functions, denoted as $\mathbb{P}_1$, associated with $\mathcal{M}_h$. We consider the following discrete spaces:
\begin{equation*}
S_h =\{\phi_h \in C(\overline{\Omega}): \phi_h|_K \in \mathbb{P}_1(K) ~ \forall K \in \mathcal{M}_h\},
\end{equation*}
where $S_h^0 = S_h \cap H_{\partial D}^1$ and $S_h^{\Omega} = S_h \cap H_{\Omega}^1$. Both $S_h^0$ and $S_h^{\Omega}$ are subspaces of $S_h$ that have vanishing degrees of freedom (DoFs) on $\partial D$ and $\Omega \cup \Gamma_R$, respectively.

\subsection{The variational formulation for IP-FEM}

We derive the variational formulation with an interior penalty term for the problem \eqref{Var1_pq}--\eqref{Var2_pq}. To facilitate the formulation, we assign a unique index $i_K \in \mathbb{N}$ to each element $K \in \mathcal{M}_h$. Furthermore, we define the jump of a function $\phi$ across an interior edge $e = \partial K \cap \partial K'$ as follows:
\begin{equation*}
[\phi]_e:=
\left\{
\begin{aligned}
& \phi|_K-\phi|_{K^{\prime}} & \text{if}~~ i_K > i_{K^{\prime}},\\
& \phi|_{K^{\prime}}-\phi|_K & \text{if}~~ i_K < i_{K^{\prime}}.
\end{aligned}
\right.
\end{equation*}
For any functions $\phi, \psi \in S_h$, we define the sesquilinear form of the interior Neumann penalty by
\begin{equation*}
J(\phi, \psi):=\sum_{e \in \mathcal C_h^I} \gamma_e h_e \langle [\partial_\nu \phi],[\partial_\nu \psi]\rangle_e,
\end{equation*}
where $h_e$ is the length of interior edge $e$ and $\gamma_e$ is a real positive parameter.

The sesquilinear forms $b^h_{1}:S_h \times S_h^0 \rightarrow \mathbb{C}$ and $b^h_{2}:S_h \times S_h^0 \rightarrow \mathbb{C}$ are defined by
\begin{align*}
b^h_{1}(\phi, \psi)= b_1(\phi, \psi) - J(\phi, \psi) ,\quad
b^h_{2}(\phi, \psi)= b_2(\phi, \psi) + J(\phi, \psi) .
\end{align*}

It is important to note that the sign of the penalty term should be consistent with that of the lower-order term (i.e., the mass matrix term) in the variational formulation. This consistency ensures enhanced stability of the solution for discrete systems constructed using linear finite elements, from a numerical computational perspective.

The variational formulation with an interior penalty term for problem \eqref{Var1_pq}--\eqref{Var2_pq} is to find
$\{p_h,q_h\} \in S_h \times S_h$ such that
\begin{equation}\label{var3_pq}
\left\{
\begin{aligned}
& b^h_{1}(p_h,\phi_h) = -\langle g_1,\phi_h\rangle_{\Gamma_R}&\forall ~ \phi_h \in S_h^0,\\
& b^h_{2}(q_h,\psi_h) = 0&\forall ~ \psi_h \in S_h^0,
\end{aligned}
\right.
\end{equation}
and
\begin{equation}\label{var4_pq}
(\nabla(p_h-q_h), \nabla\varphi_h) - \kappa^2(p_h+q_h, \varphi_h)
- J(p_h+q_h, \varphi_h) = 0 \quad \forall ~\varphi_h \in S_h^{\Omega},
\end{equation}
where $p_h = p^h_{0} + p^h_{D}$ and $q_h = q^h_{0} + p^h_{D}$ with $p^h_{0}, q^h_{0} \in S_h^0$ and $p^h_{D} \in S_h^{\Omega}$.

\subsection{The variational formulation for BP-FEM}

We now establish the variational formulation with a boundary penalty term for the problem \eqref{Var1_pq}--\eqref{Var2_pq}. Consider any functions $\phi, \psi \in S_h^{\Omega}$, the sesquilinear form of the boundary penalty term is defined as:
\begin{equation*}
G(\phi, \psi):=\sum_{e \in\mathcal C_h^B} \eta_e h_e \langle \partial_{\tau} \phi ,\partial_{\tau} \psi \rangle_e,
\end{equation*}
where $h_e$ represents the length of the boundary edge $e$, $\eta_e$ is a positive real parameter, and $\tau$ is the unit tangent vector on the boundary edge $e$.

The variational formulation with a boundary penalty term for problem \eqref{Var1_pq}--\eqref{Var2_pq} is defined as follows: find $\{p_h,q_h\} \in S_h \times S_h$ such that
\begin{equation}\label{var5_pq}
\left\{
\begin{aligned}
& b_1(p_h,\phi_h) = -\langle g_1,\phi_h\rangle_{\Gamma_R}& \forall ~ \phi_h \in S_h^0,\\
& b_2(q_h,\psi_h) = 0& \forall ~ \psi_h \in S_h^0,
\end{aligned}
\right.
\end{equation}
and
\begin{equation}\label{var6_pq}
(\nabla(p_h-q_h), \nabla\varphi_h) - \kappa^2(p_h+q_h,\varphi_h) - G(p^h_{D},\varphi_h)
= 0\quad \forall ~\varphi_h \in S_h^{\Omega},
\end{equation}
where $p_h = p^h_{0} + p^h_{D}$ and $q_h = q^h_{0} + p^h_{D}$ with $p^h_{0}, q^h_{0} \in S_h^0$ and $p^h_{D} \in S_h^{\Omega}$.

Similarly, it is essential to ensure the consistency of the sign of the penalty term $G(p^h_{D}, \varphi_h)$ with that of the lower-order term $(p_h+q_h,\varphi_h)$ in the variational formulation with a boundary penalty.

\subsection{The discretized problems}

Next, we proceed to discretize the variational problem with the interior penalty term \eqref{var3_pq}--\eqref{var4_pq} and the boundary penalty term \eqref{var5_pq}--\eqref{var6_pq} using linear FEM. Subsequently, we express these equations in matrix form.

Let $\{\alpha_j\}_{j=1}^{N_h^I}$ and $\{\beta_j\}_{j=1}^{N_h^T}$ be sets of bases in the space $S_h^0$. In the case of piecewise linear triangular elements, $N_h^I$ and $N_h^T$ correspond to the number of mesh nodes in the interior of $\Omega$ and on the boundary $\Gamma_R$, respectively. Let $\{\zeta_j\}_{j=1}^{N_h^D}$ represent the set of basis functions for the space $S_h^{\Omega}$, where $N_h^D$ denotes the number of mesh nodes on the boundary $\partial D$.

The discretized formulations of \eqref{var3_pq}--\eqref{var4_pq} and \eqref{var5_pq}--\eqref{var6_pq} for the IP-FEM and BP-FEM, using linear triangular elements, can be expressed as:
\begin{align}
\label{Discret_Mat_EQ1} \big(\mathbf{K} - \kappa^2 \mathbf{M} - \gamma \mathbf{K}_J - \mathbf{K}^{tb}\big)\mathbf{W} &= \mathbf{F},\\
\label{Discret_Mat_EQ2} \big(\mathbf{K} - \kappa^2 \mathbf{M} - \eta \mathbf{K}_G - \mathbf{K}^{tb}\big)\mathbf{W}& = \mathbf{F}.
\end{align}
In these equations, the penalty parameters $\gamma_e$ and $\eta_e$ are selected as $\gamma_e=\gamma$ for all interior edges and $\eta_e=\eta$ for all boundary edges, respectively. The unknown nodal vector $\mathbf{W}$ has a dimension of $2N_h^I + 2N_h^T + N_h^D$, given by
\begin{align*}
\mathbf{W} =
\begin{bmatrix}
\mathbf{W}_I \\
\mathbf{W}_T \\
\mathbf{P}_D
\end{bmatrix},\quad
\mathbf{W}_I =
\begin{bmatrix}
\mathbf{P}_I \\
\mathbf{Q}_I
\end{bmatrix},\quad
\mathbf{W}_T =
\begin{bmatrix}
\mathbf{P}_T \\
\mathbf{Q}_T
\end{bmatrix},
\end{align*}
where $\mathbf{P}_I$ and $\mathbf{Q}_I$ represent the values of $p$ and $q$ at the interior nodes, $\mathbf{P}_T$ and $\mathbf{Q}_T$ denote the values of $p$ and $q$ at the nodes on $\Gamma_R$, and $\mathbf{P}_D$ corresponds to the unknown nodal vector associated with the cavity boundary $\partial D$.

The stiffness matrix $\mathbf{K}$ and the mass matrix $\mathbf{M}$ are given in blockwise form as follows:
\begin{align*}
\mathbf{K} =
\begin{bmatrix}
\mathbf{K}_{II} & \mathbf{K}_{IT} & \mathbf{K}_{ID}\\
\mathbf{K}_{TI} & \mathbf{K}_{TT} & \mathbf{0}\\
\mathbf{K}_{DI} & \mathbf{0} & \mathbf{0}
\end{bmatrix}, \quad
\mathbf{M} =
\begin{bmatrix}
\mathbf{M}_{II} & \mathbf{M}_{IT} & \mathbf{M}_{ID}\\
\mathbf{M}_{TI} & \mathbf{M}_{TT} & \mathbf{0}\\
\mathbf{M}_{DI} & \mathbf{0} & 2\overline{\mathbf{M}}_{DD}
\end{bmatrix},
\end{align*}
where
\begin{align*}
& \mathbf{K}_{ij} =
\begin{bmatrix}
\overline{\mathbf{K}}_{ij} & \mathbf{0}\\
\mathbf{0} & \overline{\mathbf{K}}_{ij}
\end{bmatrix}, \quad
\mathbf{M}_{ij} =
\begin{bmatrix}
\overline{\mathbf{M}}_{ij} & \mathbf{0}\\
\mathbf{0} & -\overline{\mathbf{M}}_{ij}
\end{bmatrix},\quad i,j = I,T, \\
& \mathbf{K}_{ID} =
\begin{bmatrix}
\overline{\mathbf{K}}_{ID} \\
\overline{\mathbf{K}}_{ID}
\end{bmatrix},\quad
\mathbf{M}_{ID} =
\begin{bmatrix}
\overline{\mathbf{M}}_{ID} \\
-\overline{\mathbf{M}}_{ID}
\end{bmatrix}, \\
& \mathbf{K}_{DI} =
\begin{bmatrix}
\overline{\mathbf{K}}_{DI} & -\overline{\mathbf{K}}_{DI}
\end{bmatrix},\quad
\mathbf{M}_{DI} =
\begin{bmatrix}
\overline{\mathbf{M}}_{DI} & \overline{\mathbf{M}}_{DI}
\end{bmatrix}.
\end{align*}
Specifically, the stiffness and mass matrices associated with the IP-FEM and the BP-FEM are given in Table \ref{table_D_matrix}, where we have the relationships $\overline{\mathbf{K}}_{DI} = (\overline{\mathbf{K}}_{ID})^\top$,
$\overline{\mathbf{K}}_{TI} = (\overline{\mathbf{K}}_{IT})^\top$, $\overline{\mathbf{M}}_{DI} = (\overline{\mathbf{M}}_{ID})^\top$, and
$\overline{\mathbf{M}}_{TI} = (\overline{\mathbf{M}}_{IT})^\top$. Here $I, T$, and $D$ stand for the interior node in $\Omega$, the boundary node on $\Gamma_R$, and the boundary node on $\partial D$, respectively.

\begin{table}[ht]
\centering
\caption{The stiffness and mass matrices for the IP-FEM and BP-FEM.}
\vspace{-0.2cm}
\renewcommand\arraystretch{1.8}
\begin{tabular}{|m{1.8cm}<{\centering}|m{1.8cm}<{\centering}|m{12cm}<{\centering}|}
\hline
Matrix            & Dimension            & Matrix entries           \\
\hline
$\overline{\mathbf{K}}_{II}, \overline{\mathbf{M}}_{II}$ &
$N_h^I \times N_h^I$ &$(\overline{\mathbf{K}}_{II})_{j,l}=\displaystyle\int_{\Omega}\nabla\alpha_j \cdot \nabla\alpha_l d\boldsymbol x,\quad
            (\overline{\mathbf{M}}_{II})_{j,l}=\displaystyle\int_{\Omega}\alpha_j \alpha_l d\boldsymbol x$\\
$\overline{\mathbf{K}}_{IT}, \overline{\mathbf{M}}_{IT}$ &
$N_h^I \times N_h^T$ &
$(\overline{\mathbf{K}}_{IT})_{j,l}=\displaystyle\int_{\Omega}\nabla\alpha_j \cdot \nabla\beta_l d\boldsymbol x, \quad
            (\overline{\mathbf{M}}_{IT})_{j,l}=\displaystyle\int_{\Omega} \alpha_j \beta_l d\boldsymbol x$ \\
$\overline{\mathbf{K}}_{ID}, \overline{\mathbf{M}}_{ID}$ &
$N_h^I \times N_h^D$ &
$(\overline{\mathbf{K}}_{ID})_{j,l}=\displaystyle\int_{\Omega}\nabla\alpha_j \cdot \nabla\zeta_l d\boldsymbol x,\quad
            (\overline{\mathbf{M}}_{ID})_{j,l}=\displaystyle\int_{\Omega} \alpha_j \zeta_l d\boldsymbol x $ \\
$\overline{\mathbf{K}}_{TT},\overline{\mathbf{M}}_{TT}$ &
$N_h^T \times N_h^T$ &
$(\overline{\mathbf{K}}_{TT})_{j,l}=\displaystyle\int_{\Omega}\nabla\beta_j \cdot \nabla\beta_l d\boldsymbol x,\quad
            (\overline{\mathbf{M}}_{TT})_{j,l}=\displaystyle\int_{\Omega}\beta_j \beta_l d\boldsymbol x$ \\
\rule{0pt}{16pt} $\overline{\mathbf{M}}_{DD}$ &
$N_h^D \times N_h^D$ &
$(\overline{\mathbf{M}}_{DD})_{j,l}=\displaystyle\int_{\Omega}\zeta_j \zeta_l d\boldsymbol x$ \\[5pt]
\hline
\end{tabular}
\label{table_D_matrix}
\end{table}

The matrix $\mathbf{K}^{tb}$ is associated with the TBC and is given by
\begin{align*}
\mathbf{K}^{tb} =
\begin{bmatrix}
\mathbf{0} & \mathbf{0} & \mathbf{0}\\
\mathbf{0} & \mathbf{K}_{L}^{tb} & \mathbf{0}\\
\mathbf{0} & \mathbf{0} & \mathbf{0}
\end{bmatrix},
\end{align*}
where the matrix $\mathbf{K}_L^{tb}$ can be given by
\begin{align*}
\mathbf{K}_L^{tb} = \sum_{|n| \leq N}
\begin{bmatrix}
a_n \mathbf{K}_n^{tb} & 0\\
0 & b_n \mathbf{K}_n^{tb}
\end{bmatrix}.
\end{align*}
Here the truncation parameter $N$ is a positive integer, $a_n = h_n(\kappa R)/2\pi$, $ b_n = k_n(\kappa R)/2\pi$,
and $\mathbf{K}_n^{tb}$ is evaluated as follows:
\begin{align*}
\mathbf{K}_n^{tb} = \left(\int_0^{2\pi}\mathbf{N}^{tb}e^{{\rm i}n\theta}d\theta\right)_{N_{h}^T \times 1}
             \left(\int_0^{2\pi}(\mathbf{N}^{tb})^Te^{-{\rm i}n\theta'}d\theta'\right)_{1 \times N_{h}^T},
\end{align*}
where $\mathbf{N}^{tb}$ is a vector consisting of base functions $\{\beta_j\}_{j=1}^{N_h^T}$ on the boundary $\Gamma_R$.
The matrix $\overline{\mathbf{K}}_J$ associated with the interior penalty term $J(\phi, \psi)$ is given by
\begin{equation*}
\overline{\mathbf{K}}_J =
\begin{bmatrix}
\overline{\mathbf{J}}_{II} & \overline{\mathbf{J}}_{IT} & \overline{\mathbf{J}}_{ID}\\
\overline{\mathbf{J}}_{TI} & \overline{\mathbf{J}}_{TT} & 0              \\
\overline{\mathbf{J}}_{DI} & 0               & \overline{\mathbf{J}}_{DD}
\end{bmatrix} =
\sum_{j=\partial K \cap \partial K^{\prime} \in \mathcal{C}_h} h_j^2\mathbf{k}_j,
\end{equation*}
where $\overline{\mathbf{J}}_{TI} = (\overline{\mathbf{J}}_{IT})^{\top}$, $\overline{\mathbf{J}}_{DI} = (\overline{\mathbf{J}}_{ID})^{\top}$,
$\mathbf{k}_j = \mathbf{g}_j \mathbf{g}_j^\top$ with $\mathbf{g}_j$ being the discretized vector associated with the jump $[\partial_\nu \phi]$.
Specifically, the jumps $[\partial_\nu \phi]$ and $[\partial_\nu \psi]$ on the interior edge $j=\partial K \cap \partial K^{\prime}$ in the interior penalty term $J(\phi, \psi)$ can be written as
\[
[\partial_\nu \phi] = \partial_\nu^K \phi^K + \partial_\nu^{K^{\prime}} \phi^{K^{\prime}} = \mathbf{g}_j^\top \boldsymbol{\phi}, \quad
[\partial_\nu \psi] = \partial_\nu^K \psi^K + \partial_\nu^{K^{\prime}} \psi^{K^{\prime}} = \mathbf{g}_j^\top \boldsymbol{\psi},
\]
where $\boldsymbol{\phi}$ and $\boldsymbol{\psi}$ are vectors composed of the function values of $\phi$ and $\psi$ at all nodes in the domain $\Omega$, respectively. Additionally, the normal direction in $\partial_\nu^K$ is opposite to that in $\partial_\nu^{K^{\prime}}$.

Therefore, the interior penalty stiffness matrix $\mathbf{K}_J$ can be expressed as
\begin{align*}
\mathbf{K}_J =
\begin{bmatrix}
\mathbf{J}_{II} & \mathbf{J}_{IT} & \mathbf{J}_{ID}\\
\mathbf{J}_{TI} & \mathbf{J}_{TT} & \mathbf{0}\\
\mathbf{J}_{DI} & \mathbf{0} & 2\overline{\mathbf{J}}_{DD}
\end{bmatrix},
\end{align*}
where
\begin{align}
& \mathbf{J}_{ij} =
\begin{bmatrix}
\overline{\mathbf{J}}_{ij} & \mathbf{0}\\
\mathbf{0} & -\overline{\mathbf{J}}_{ij}
\end{bmatrix},i,j=I,T,\quad
\mathbf{J}_{ID} =
\begin{bmatrix}
\overline{\mathbf{J}}_{ID} \\
-\overline{\mathbf{J}}_{ID}
\end{bmatrix},\quad
\mathbf{J}_{DI} =
\begin{bmatrix}
\overline{\mathbf{J}}_{DI} & \overline{\mathbf{J}}_{DI}
\end{bmatrix}.\notag
\end{align}

Let us assume that the boundary $\partial D$ of the cavity is divided into $K$ segments $\Gamma_k$, where $k=1,\cdots, K$, in the mesh $\mathcal{M}_h$. The boundary penalty stiffness matrix $\mathbf{K}_G$ is associated with the boundary penalty term $G(p_h^{D},\varphi_h)$ and can be obtained by mapping $\mathbf{K}_G^L$ from local to global numbering. The matrix $\mathbf{K}_G^L$ can be evaluated as follows:
\begin{align*}
\mathbf{K}_G^L = \sum_{k=1}^K
\begin{bmatrix}
1 & -1\\
-1 & 1
\end{bmatrix}.
\end{align*}

\section{Numerical experiments}\label{sec6}

 In this section, we present numerical results obtained using the IP-FEM and BP-FEM for three examples: a circular-shaped cavity, an ellipse-shaped cavity, and a kite-shaped cavity. In the experiments, we investigate the out-of-plane displacement of the scattered field $v$ and its bending moment $w=\kappa^{-2}\Delta v$ by solving the boundary value problem \eqref{PQ_bvp_Scattered} and using the relationships $(p^s, q^s)$ and $(v, \Delta v)$. The relative errors in the $\emph{L}^2$ norm and the $\emph{H}^1$ semi-norm are employed to assess the numerical solutions. For the circular-shaped cavity, we compare the results against the analytic solution, while for the ellipse-shaped and kite-shaped cavities, we use reference solutions, i.e., the numerical solutions obtained with fine meshes. The relative errors of the $\emph{L}^2$ norm and the $\emph{H}^1$ semi-norm of any function $\phi$ are defined as follows:
 \[
 \mathrm{E}_{L^2}=\frac{\|\phi^e-\phi^n\|_{0,\Omega}}{\|\phi^e\|_{0, \Omega}},\quad \mathrm{E}_{H^1}=\frac{\|\nabla\phi^e-\nabla\phi^n\|_{0,\Omega}}{\|\nabla\phi^e\|_{0, \Omega}},
 \]
where $\phi^e$ and $\phi^n$ represent the analytical or reference solution and the numerical solution, respectively.

\subsection{A circular-shaped cavity}

Consider a circular-shaped cavity $D=B_{\hat R}$, which is illuminated by a plane wave
\[
u^{\rm inc}(\boldsymbol{x}) = e^{{\rm i}\kappa \boldsymbol{x}\cdot \boldsymbol{d}},
\]
where $\kappa >0 $ is the wavenumber and $\boldsymbol{d} = (\cos\alpha, \sin \alpha)$ is the incident direction with $\alpha$ being the incident angle. The parameter equation of the circular-shaped cavity with radius $\hat R$ is
\begin{align*}
x(\theta) = \hat R \cos\theta,\quad y(\theta) = \hat R\sin\theta,\quad \theta\in [0, 2\pi).
\end{align*}

\subsubsection{The analytical solution}

The Helmholtz and modified Helmholtz wave components $v_H$ and $v_M$ of the out-of-plane displacement of the scattered field $v$ satisfy the coupled boundary value problem
\begin{equation} \label{SF_bvp3}
\left\{
\begin{aligned}
& \Delta v_H + \kappa^2 v_H = 0,\quad   \Delta v_M - \kappa^2 v_M = 0& \text{in} ~ \mathbb{R}^2\setminus\overline{B_{\hat R}},\\
& v_H + v_M = f(\theta),\quad \partial_r v_H + \partial_r v_M = g(\theta) & \text{on} ~ \partial B_{\hat R},\\
\end{aligned}
\right.
\end{equation}
where $f(\theta) = -u^{\rm inc}$ and $g(\theta)=-\partial_r u^{\rm inc}$.
The analytical solution of \eqref{SF_bvp3} has the Fourier series expansion in polar coordinates:
\begin{equation}\label{FourSF}
v_H(r,\theta) = \sum_{n \in \mathbb{Z}} \frac{H_n^{(1)}(\kappa r)}{H_n^{(1)}(\kappa \hat R)}v_H^{(n)}(\hat R)e^{{\rm i} n\theta}, \quad
v_M(r,\theta) = \sum_{n \in \mathbb{Z}} \frac{K_n(\kappa r)}{K_n(\kappa \hat R)}v_M^{(n)}(\hat R)e^{{\rm i} n\theta},
\end{equation}
where the Fourier coefficients $v_H^{(n)}(\hat R)$ and $v_M^{(n)}(\hat R)$ are given by
\[
 v_H^{(n)}(\hat R) = \frac{1}{2\pi}\int_0^{2\pi} v_H(\hat R,\theta) e^{-{\rm i}n\theta} d\theta,\quad v_M^{(n)}(\hat R) = \frac{1}{2\pi}\int_0^{2\pi} v_M(\hat R,\theta) e^{-{\rm i}n\theta} d\theta.
\]
Since $f(\theta)$ and $g(\theta)$ are periodic functions with period $2\pi$, we have
\begin{align}\label{Fourbd}
f(\theta) = \sum\limits_{n \in \mathbb{Z}}f^{(n)}e^{{\rm i} n \theta}, \quad
g(\theta) = \sum\limits_{n \in \mathbb{Z}}g^{(n)}e^{{\rm i} n \theta},
\end{align}
where the Fourier coefficients $f^{(n)}$ and $g^{(n)}$ are
\[
f^{(n)} = \frac{1}{2\pi}\int_0^{2\pi}f(\theta) e^{-{\rm i}n\theta} d\theta, \quad g^{(n)} = \frac{1}{2\pi}\int_0^{2\pi}g(\theta) e^{-{\rm i}n\theta} d\theta.
\]
Substituting \eqref{FourSF}--\eqref{Fourbd} into the boundary condition on $\partial B_{\hat R}$ yields a linear system of algebraic equations
\begin{equation*}
\left\{
\begin{aligned}
& v_H^{(n)}(\hat R) + v_M^{(n)}(\hat R) = f^{(n)},\\
& \kappa \frac{H_n^{(1)'}(\kappa \hat R)}{H_n^{(1)}(\kappa \hat R)}v_H^{(n)} +
    \kappa \frac{K_n^{'}(\kappa \hat R)}{K_n(\kappa \hat R)}v_M^{(n)} =g^{(n)},
\end{aligned}
\right.
\end{equation*}
which has a matrix form
\begin{align} \label{ExactS}
A
    \begin{bmatrix}
        v_H^{(n)}\\
        v_M^{(n)}
    \end{bmatrix} =
    \begin{bmatrix}
        1 & 1\\
        \kappa\frac{H_n^{(1)'}(\kappa \hat R)}{H_n^{(1)}(\kappa \hat R)}
        & \kappa \frac{K_n^{'}(\kappa \hat R)}{K_n(\kappa \hat R)}
    \end{bmatrix}
    \begin{bmatrix}
        v_H^{(n)}\\
        v_M^{(n)}
    \end{bmatrix}=
    \begin{bmatrix}
        f^{(n)}\\
        g^{(n)}
    \end{bmatrix}.
\end{align}
We can obtain the solution of \eqref{ExactS} using Cramer's rule that
\begin{align}\label{FourCo}
\left\{
\begin{aligned}
& v_H^{(n)}  = \frac{1}{b_n}\left(\kappa^{-1} g^{(n)}
            -\frac{K_n^{'}(\kappa \hat R)}{K_n(\kappa \hat R)}f^{(n)}\right),\\
& v_M^{(n)}  = \frac{1}{b_n}\left(\frac{H_n^{(1)'}(\kappa \hat R)}{H_n^{(1)}(\kappa \hat R)}f^{(n)}
            -\kappa^{-1} g^{(n)}\right),
\end{aligned}
\right.
\end{align}
where $b_n$ is the determinant of the coefficient matrix $A$ and is given by
\[
b_n= \left(\frac{H_n^{(1)'}(\kappa \hat R)}{H_n^{(1)}(\kappa \hat R)}
                - \frac{K_n^{'}(\kappa \hat R)}{K_n(\kappa \hat R)}\right).
\]

From \eqref{FourSF} and \eqref{FourCo}, we can obtain the analytical solutions $v_H$ and $v_M$. Then, using the following relationships, the scattered field $v$ and its bending moment $w$ can be expressed as follows:
\[
v = v_H + v_M, \quad w = v_M - v_H.
\]

\begin{theorem}\label{helmphi}
The linear system \eqref{ExactS} has a unique solution.
\end{theorem}

\begin{proof}
It suffices to show that the coefficient matrix $A$ of \eqref{ExactS} is nonsingular, i.e., $\det(A)\neq 0$. A simple calculation gives
\begin{align*}
\det(A) = \frac{\kappa K_n{'}(\kappa \hat R)}{K_n(\kappa \hat R)}
        - \frac{\kappa H_n^{(1)'}(\kappa \hat R)}{H_n^{(1)}(\kappa \hat R)}
        = \frac{1}{\hat{R}}\left(k_n(\kappa\hat{R}) - h_n(\kappa\hat{R})\right).
\end{align*}
Taking the imaginary part of $\det(A)$ and using Lemma \ref{Lemma1}, we have
\[
\Im(\det(A)) = -\frac{1}{\hat R} \Im\{h_n(\kappa \hat{R})\}
= -\frac{2}{\pi \hat{R}}\frac{1}{|H_n^{(1)}(\kappa \hat{R})|^2}\neq 0,
\]
which implies that the coefficient matrix $A$ of \eqref{ExactS} is nonsingular and there exists a unique solution to the system of equations \eqref{ExactS}.
\end{proof}

In the experiments, we set $\hat R = 0.3$ for the circular-shaped cavity and the radius $R = 0.6$ for the TBC. The incident angle $\alpha = \pi/3$ and the wavenumber $\kappa = \pi$, corresponding to a wavelength $\lambda = 2$. We choose the DtN operator truncation number $N$ to be 15.

\subsubsection{The influence of $\gamma$}

In the IP-FEM, the penalty parameter $\gamma$ plays a crucial role. In this subsection, we investigate the influence of the penalty parameter $\gamma$ on the accuracy of the IP-FEM. If $\gamma$ is too large, it introduces artificial dissipation in the numerical results. On the other hand, if $\gamma$ is too small, we observe an oscillation behavior of the bending moment $w$ on the cavity boundary, similar to what is seen in the regular linear FEM ($\gamma=0$), as shown in the left part of Figure \ref{Figure_Gamma_BC}.

\begin{figure}
  \centering
  \includegraphics[width=0.45\textwidth]{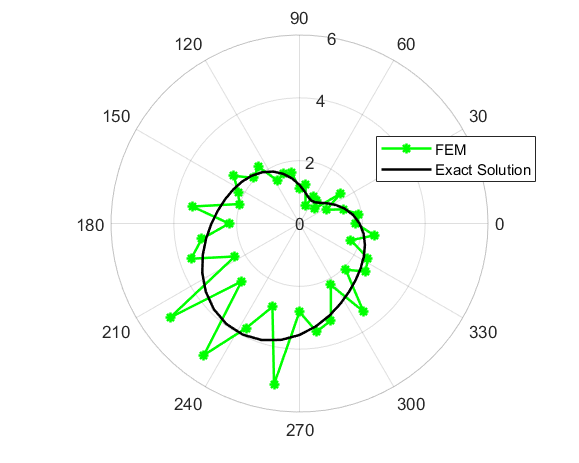}
  \includegraphics[width=0.45\textwidth]{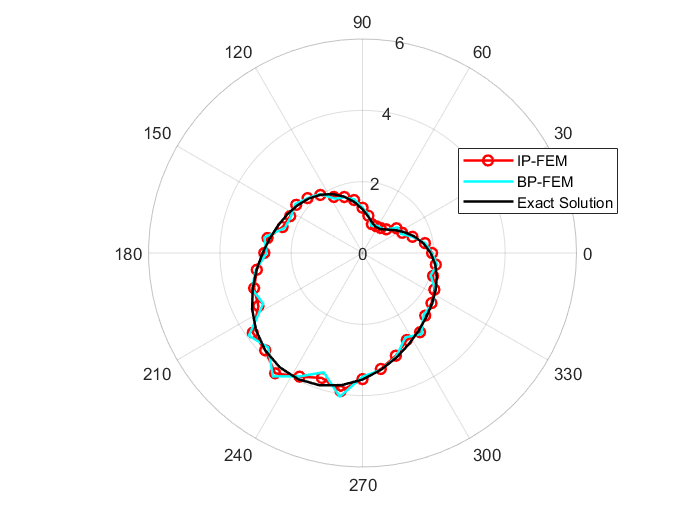}
  \caption{Example 1: The bending moment $w$ on the boundary of cavity:
  (left) The regular linear FEM ($\gamma=0$); (right) The IP-FEM ($\gamma = 4.2\times 10^{-3}$) and the BP-FEM ($\eta = 2.5\kappa\times 10^{-3} $).}
  \label{Figure_Gamma_BC}
\end{figure}

Figure \ref{Figure_gamma} presents the relative errors of the $L^2$ norm and the $H^1$ semi-norm of $v$ and $w$ with different values of the parameter $\gamma$ at the mesh size $h=0.05$. The relative $L^2$ and $H^1$ errors of $v$ increase as $\gamma$ becomes larger, but they remain at levels of $10^{-3}$ and $10^{-2}$, respectively. However, the relative $L^2$ and $H^1$ errors of $w$ first decrease and then increase with increasing $\gamma$, with the minimum relative $L^2$ error of $w$ occurring at $\gamma=4.2\times 10^{-3}$. The solution of $w$ on the cavity boundary for the IP-FEM with the optimal parameter $\gamma=4.2\times 10^{-3}$ is shown in the right part of Figure \ref{Figure_Gamma_BC}. It is evident that the boundary oscillation behavior of $w$ is mitigated compared to the regular linear FEM, i.e., the linear FEM without any penalty term ($\gamma=0$). Based on these observations, we conclude that there exists a range of values for $\gamma$ that significantly improves the results for the bending moment $w$ while maintaining good results for the displacement $v$.

\begin{figure}
  \centering
  \includegraphics[width=0.8\textwidth]{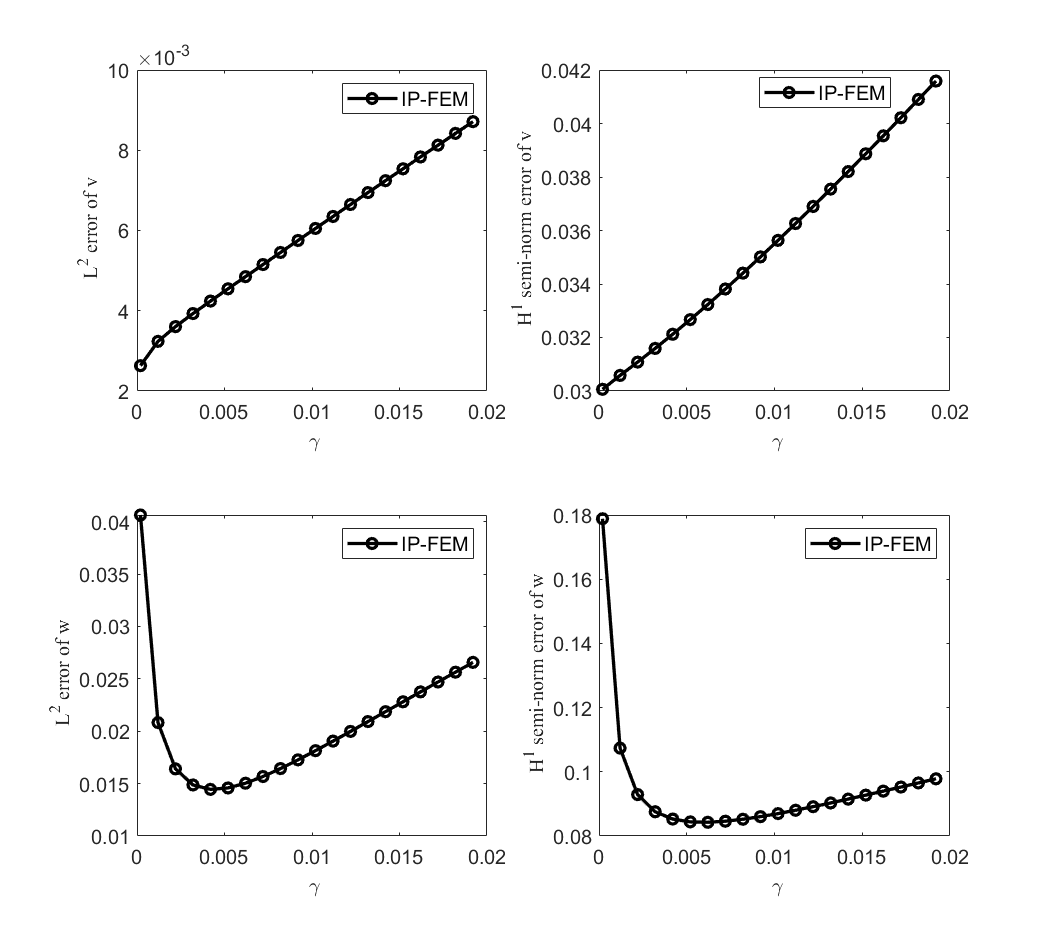}
  \caption{Example 1: The relative errors are plotted against the parameter $\gamma$ for $\kappa=\pi$.}
  \label{Figure_gamma}
\end{figure}

\begin{figure}
  \centering
  \includegraphics[width=0.92\textwidth]{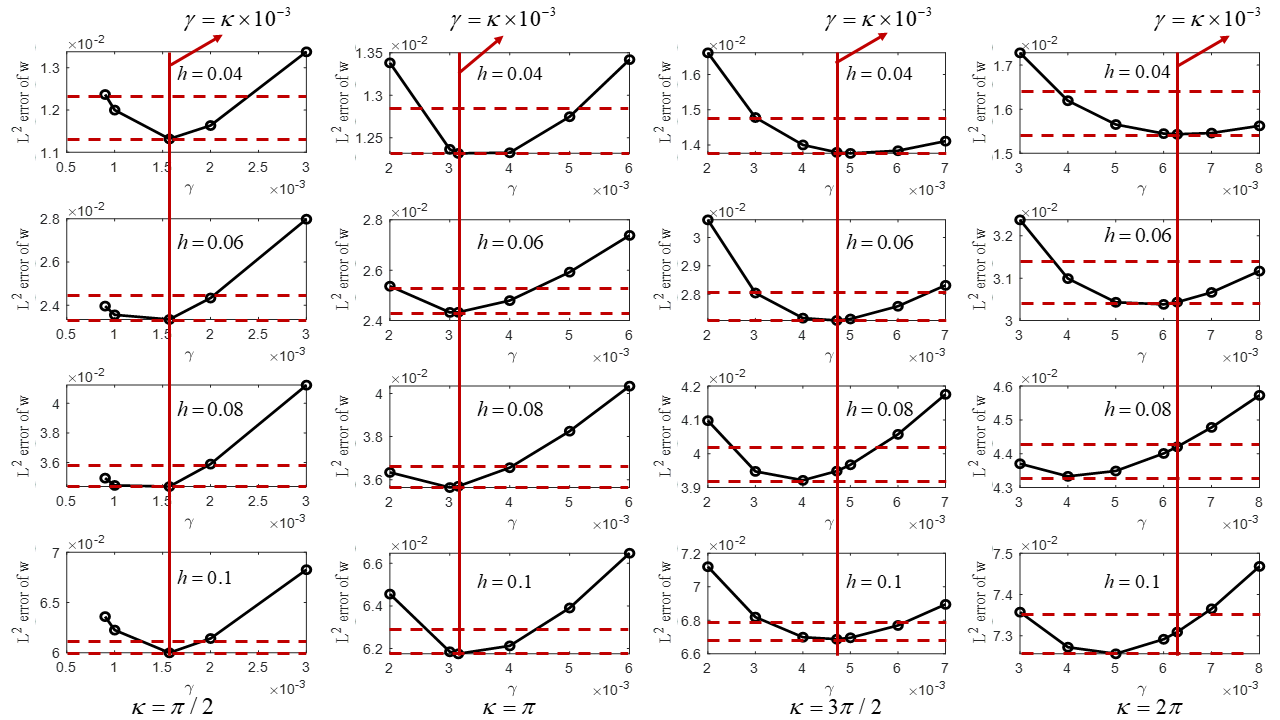}
  \caption{The optimal parameter $\gamma$ with different wavenumbers and mesh sizes $h$ for the IP-FEM.}
  \label{Figure_gamma_selection}
\end{figure}

Figure \ref{Figure_gamma_selection} presents the optimal penalty parameter $\gamma$ at different wavenumbers ($\kappa=0.5\pi, \pi, 1.5\pi, 2\pi$) and mesh sizes ($h=0.04, 0.06, 0.08, 0.1$). The figure contains sixteen cases, and each subfigure shows the variation of the relative $L^2$ error of $w$ with the parameter $\gamma$. In this paper, relative $L^2$ errors of $w$ within $1\times10^{-3}$ (shown by two horizontal red dashed lines from the smallest error $e_{\rm low}$ to $e_{\rm low}+1\times10^{-3}$) are considered acceptable, and the corresponding penalty parameters $\gamma$ are considered acceptably optimal. From these subfigures, we observe that for $\kappa=0.5\pi$, the optimal parameters are located at $\gamma=1.57\times10^{-3}$ for $h=0.04, 0.06, 0.08, 0.1$, and similar results are obtained for other cases ($\kappa = \pi, 1.5\pi, 2\pi$). This implies that the optimal parameter $\gamma$ is directly proportional to the wavenumber $\kappa$ and is less affected by the mesh size $h$. For the discretized problem \eqref{Discret_Mat_EQ1}, a suitable choice for the penalty parameter is $\gamma = \kappa \times 10^{-3} $.

\subsubsection{The influence of $\eta$}

In this subsection, we explore the influence of the penalty parameter $\eta$ in the BP-FEM. In this method, $\eta$ plays an important role as a penalty parameter. An appropriate parameter value $\eta$ can effectively suppress the oscillation of $w$ on the cavity boundary, as shown in the right part of Figure \ref{Figure_Gamma_BC}. However, if $\eta$ is too large, it introduces artificial dissipation in the numerical results. Conversely, if $\eta$ is too small, we observe oscillation behavior of the bending moment $w$ on the cavity boundary, similar to what is seen in the regular linear FEM ($\eta=0$), as shown in the left part of Figure \ref{Figure_Gamma_BC}.

Figure \ref{Figure_eta} displays the relative $L^2$ and $H^1$ errors of $v$ and $w$ for different values of the parameter $\eta$ at a mesh size of $h=0.05$. These figures reveal that the relative $L^2$ and $H^1$ errors of $v$ increase as $\eta$ increases, but they remain at levels of $10^{-3}$ and $10^{-2}$, respectively. Conversely, the relative $L^2$ errors of $w$ first decrease and then increase as $\eta$ increases, with the smallest error obtained for $\eta \in (7.0\times 10^{-3},1.7\times 10^{-2})$. Based on these observations, we conclude that there exists a range of values for $\eta$, where the results for $w$ are significantly improved while maintaining good results for $v$.

\begin{figure}
  \centering
  \includegraphics[width=0.8\textwidth]{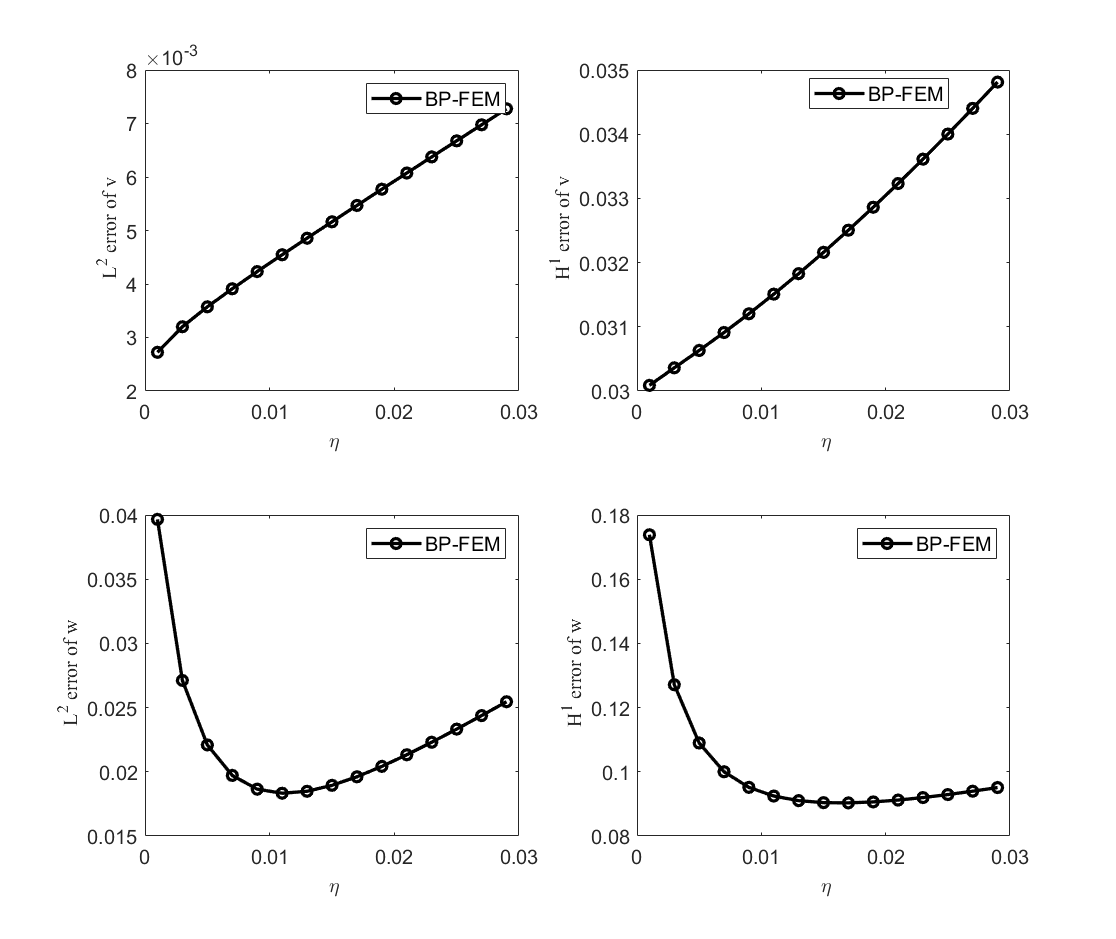}
  \caption{Example 1: The relative errors are plotted against the parameter $\eta$ for $\kappa=\pi$.}
  \label{Figure_eta}
\end{figure}

Figure \ref{Figure_eta_selection} presents the optimal penalty parameter $\eta$ at different wavenumbers ($\kappa=0.5\pi, \pi, 1.5\pi, 2\pi$) and mesh sizes ($h=0.04, 0.06, 0.08, 0.1$). The figure contains sixteen cases, and each subfigure shows the variation of the relative $L^2$ error of $w$ with the parameter $\eta$. In this paper, relative $L^2$ errors of $w$ within $1\times10^{-3}$ (shown by two horizontal red dashed lines from the smallest error $e_{\rm low}$ to $e_{\rm low}+1\times10^{-3}$) are considered acceptable, and the corresponding penalty parameters $\eta$ are considered acceptably optimal. From these subfigures, we observe similar results to those in Figure \ref{Figure_gamma_selection}. This also implies that the optimal parameter $\eta$ is directly proportional to the wavenumber $\kappa$ and is less affected by the mesh size $h$. Consequently, for convenience, the penalty parameter $\eta$ for the discretized problem \eqref{Discret_Mat_EQ2} can be chosen as $\eta = 2.5\kappa \times 10^{-3}$.

\begin{figure}
  \centering
  \includegraphics[width=0.92\textwidth]{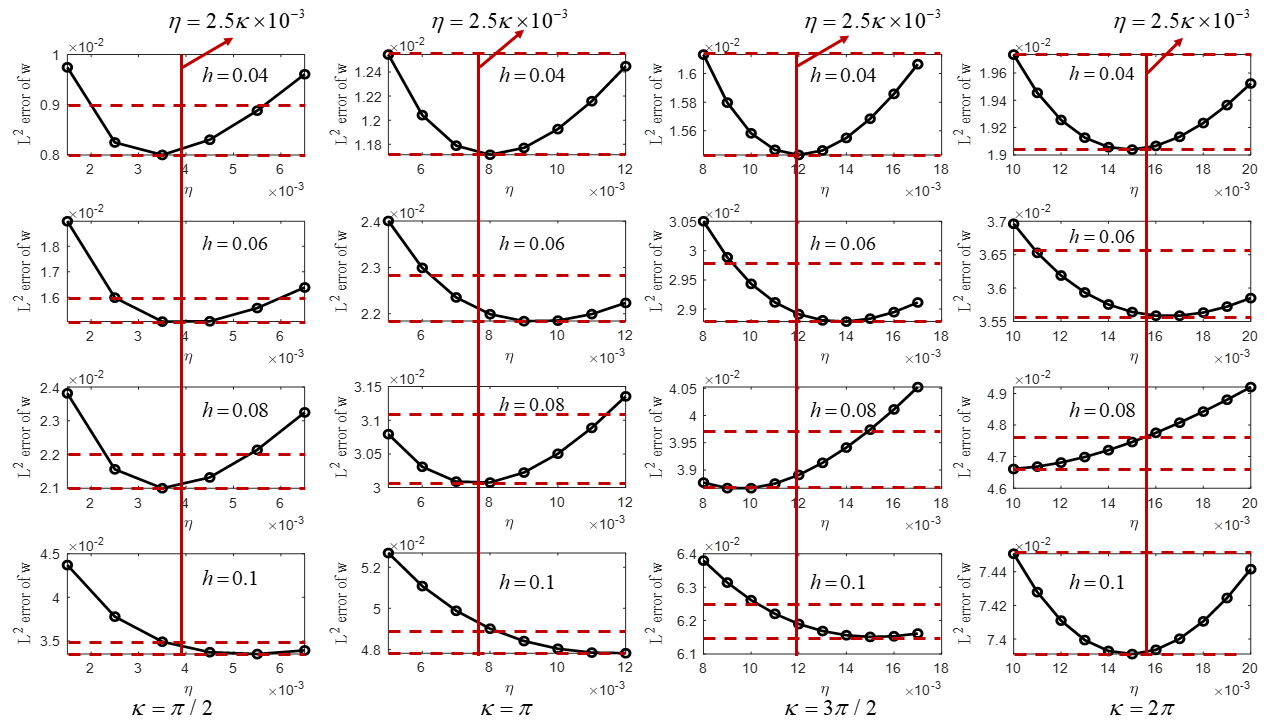}
  \caption{The optimal parameter $\eta$ with different wavenumbers and mesh sizes $h$ in the BP-FEM.}
  \label{Figure_eta_selection}
\end{figure}

\subsubsection{The influence of $\kappa$}

In this subsection, we consider the effects of the wavenumber on the solution accuracy with a fixed mesh size of $h=0.05$. Figure \ref{Figure_Ex1_wavenumber} illustrates the relative $L^2$ and $H^1$ errors of solutions $v$ and $w$ obtained using the regular linear FEM ($\gamma=0$), the IP-FEM ($\gamma = \kappa\times 10^{-3}$), and the BP-FEM ($\eta = 2.5\kappa \times 10^{-3}$). From these figures, we observe that the behavior of $v$ is similar for all three methods. However, for the solution $w$, both the IP-FEM and the BP-FEM show significant improvements compared to the regular linear FEM. Additionally, the errors in both $L^2$ and $H^1$ norms increase as the wavenumber $\kappa$ increases, regardless of the method used.

\begin{figure}
  \centering
  \includegraphics[width=0.8\textwidth]{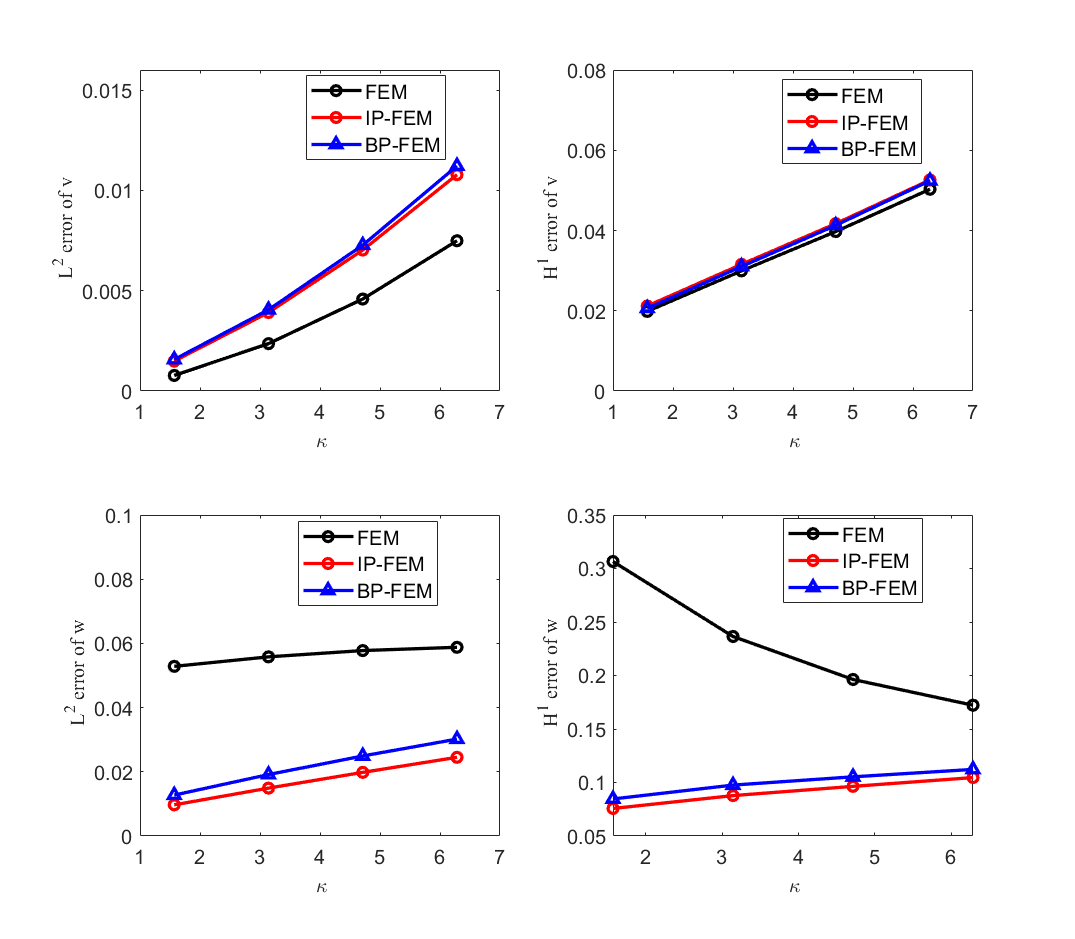}
  \caption{Example 1: The relative errors are plotted against the wavenumber with $\gamma =\kappa\times 10^{-3}$ and $\eta = 2.5\kappa\times 10^{-3} $.}\label{Figure_Ex1_wavenumber}
\end{figure}

\subsubsection{Convergence}

In this subsection, we examine the convergence of the IP-FEM and the BP-FEM. Figure \ref{Figure_Ex1_convergence} displays the relative errors of the $L^2$ norm and the $H^1$ semi-norm for the scattered field $v$ and its bending moment $w$ using different methods. From these figures, we observe that the convergence rates of the relative $L^2$ and $H^1$ errors of $v$ and $w$ for both the IP-FEM and the BP-FEM achieve the optimal convergence order.

\begin{figure}
  \centering
  \includegraphics[width=0.8\textwidth]{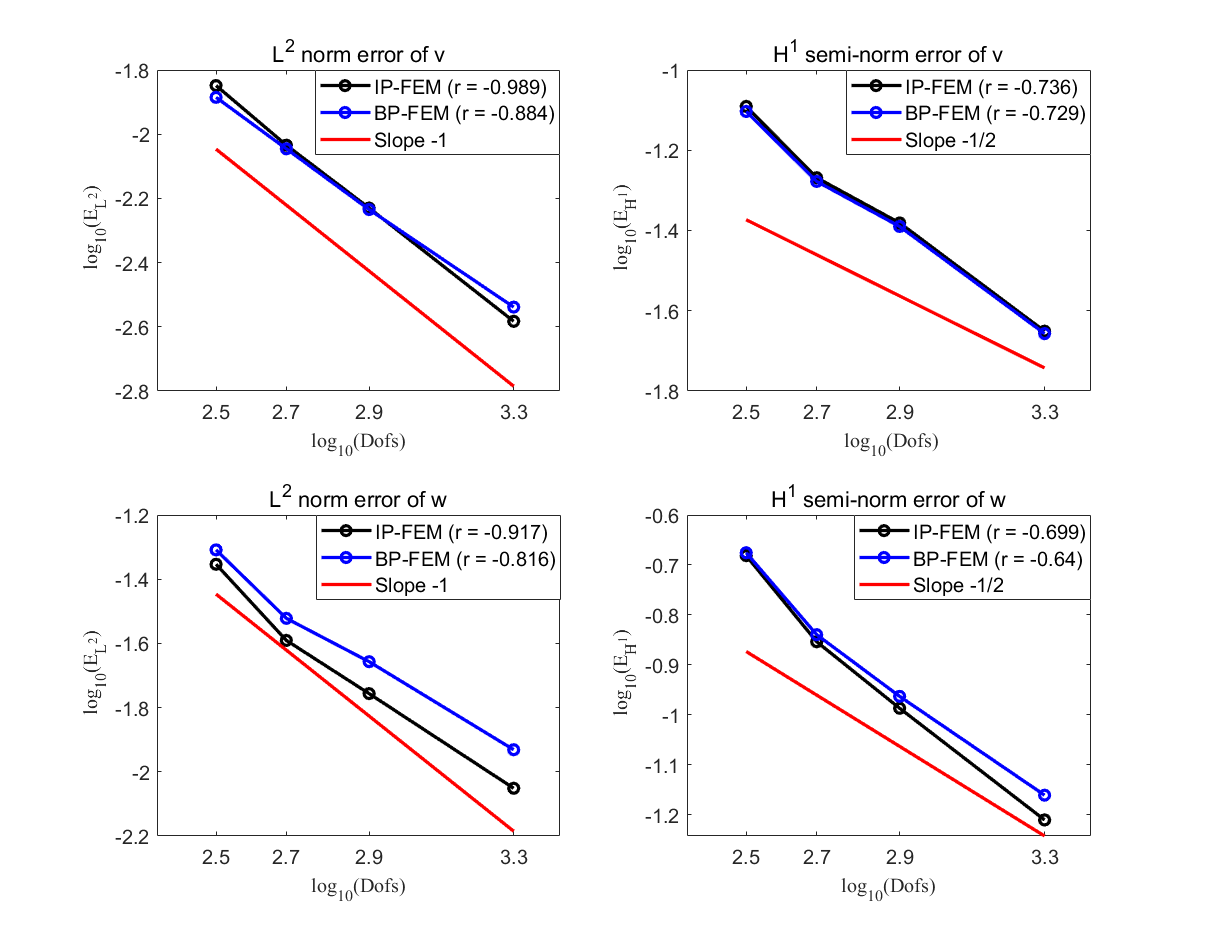}
  \caption{Example 1: The relative errors are plotted against the degrees of freedom (Dofs) for the IP-FEM ($\gamma = \pi \times 10^{-3}$) and the BP-FEM ($\eta = 2.5\pi \times 10^{-3}$).}\label{Figure_Ex1_convergence}
\end{figure}

\subsection{An ellipse-shaped cavity}

In this example, we study the flexural wave scattering by an elliptical cavity with the clamped boundary. The boundary of the ellipse is described by the following parametric equations:
\begin{equation*}
x(t)=a\cos(\theta),\quad y(t)=b\sin(\theta),
\end{equation*}
where the major semi-axis $a = 0.4$ and the minor semi-axis $b = 0.2$. The parameter $\theta$ ranges from $0$ to $2\pi$.
In the experiments, the open domain is truncated by a circle with a radius $R = 0.6$, and we choose the penalty parameters as follows: $\gamma = \kappa \times 10^{-3}$ for the IP-FEM and $\eta = 2.5\kappa \times 10^{-3} $ for the BP-FEM. All other related parameters remain the same as in the first example. For the sake of comparison, we obtain the reference solution using the IP-FEM with $\gamma = \kappa \times 10^{-3}$ on a very fine mesh.

\subsubsection{Accuracy}

In this subsection, we consider the effectiveness of the IP-FEM and the BP-FEM. The mesh size and the wavenumber are set as $h=0.05$ and $\kappa = \pi$, respectively. Figure \ref{Figure_example2_accuracy} shows the solutions $w$ obtained using the regular FEM ($\gamma=0$), the IP-FEM ($\gamma = \kappa\times 10^{-3}$), and the BP-FEM ($\eta = 2.5\kappa \times 10^{-3}$) on the cavity boundary. For the results on the entire domain, we only present the regular linear FEM and the IP-FEM, as the BP-FEM yields similar outcomes to the IP-FEM. From these figures, we observe that both the IP-FEM and the BP-FEM effectively suppress the oscillations of the bending moment $w$ on the cavity boundary when compared with the regular linear FEM.

\begin{figure}
  \centering
  \includegraphics[width=0.45\textwidth]{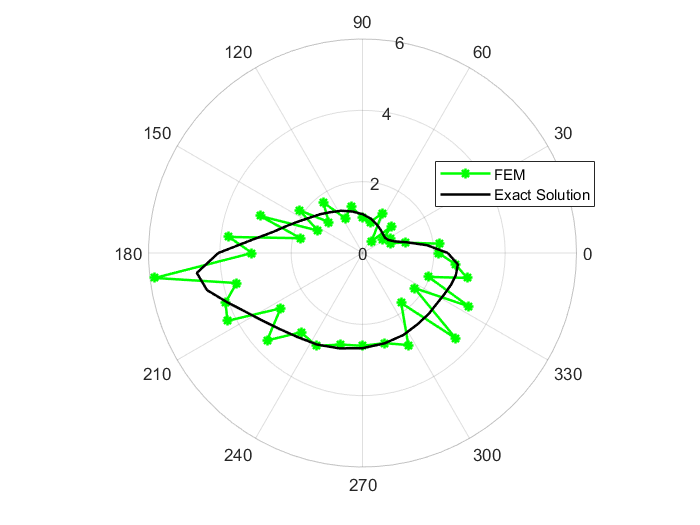}
  \includegraphics[width=0.45\textwidth]{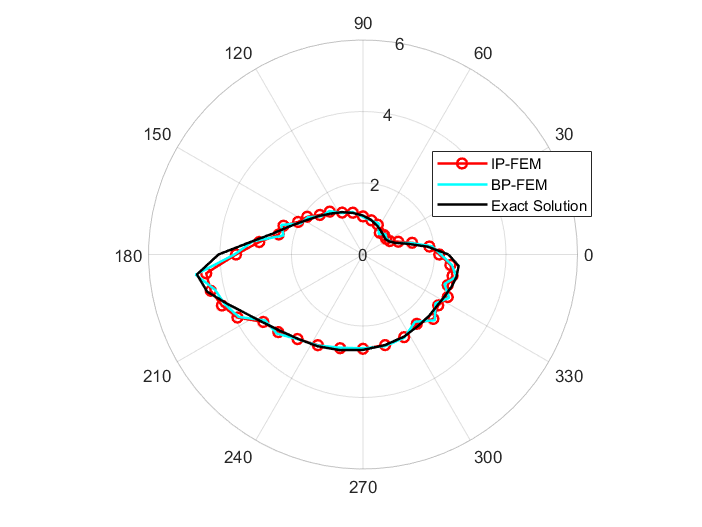}
  \includegraphics[width=0.45\textwidth]{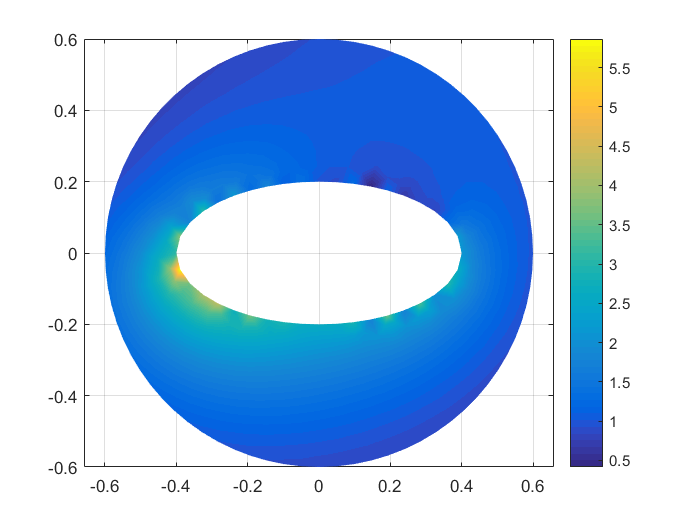}
  \includegraphics[width=0.45\textwidth]{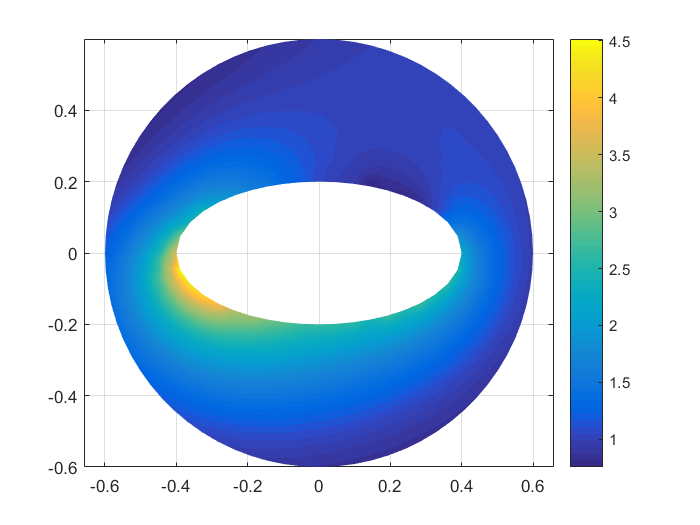}
  \caption{Example 2: The bending moment $w$ on the boundary of cavity and the entire domain:
  (left) The regular linear FEM ($\gamma=0$); (right) the IP-FEM ($\gamma = \kappa\times 10^{-3}$ and the BP-FEM ($\eta = 2.5\kappa\times 10^{-3}$).}
  \label{Figure_example2_accuracy}
\end{figure}

\subsubsection{Convergence}

The convergence of the IP-FEM and the BP-FEM is investigated in this subsection. Figure \ref{Figure_Convergence_example2} illustrates the convergence of the relative errors of the $L^2$ norm and the $H^1$ semi-norm for the scattered field $v$ and its bending moment $w$ using various methods. From these figures, it is evident that the convergence rates of the relative $L^2$ and $H^1$ errors for $v$ and $w$ with the IP-FEM and the BP-FEM achieve optimal convergence orders.

\begin{figure}
  \centering
  \includegraphics[width=0.8\textwidth]{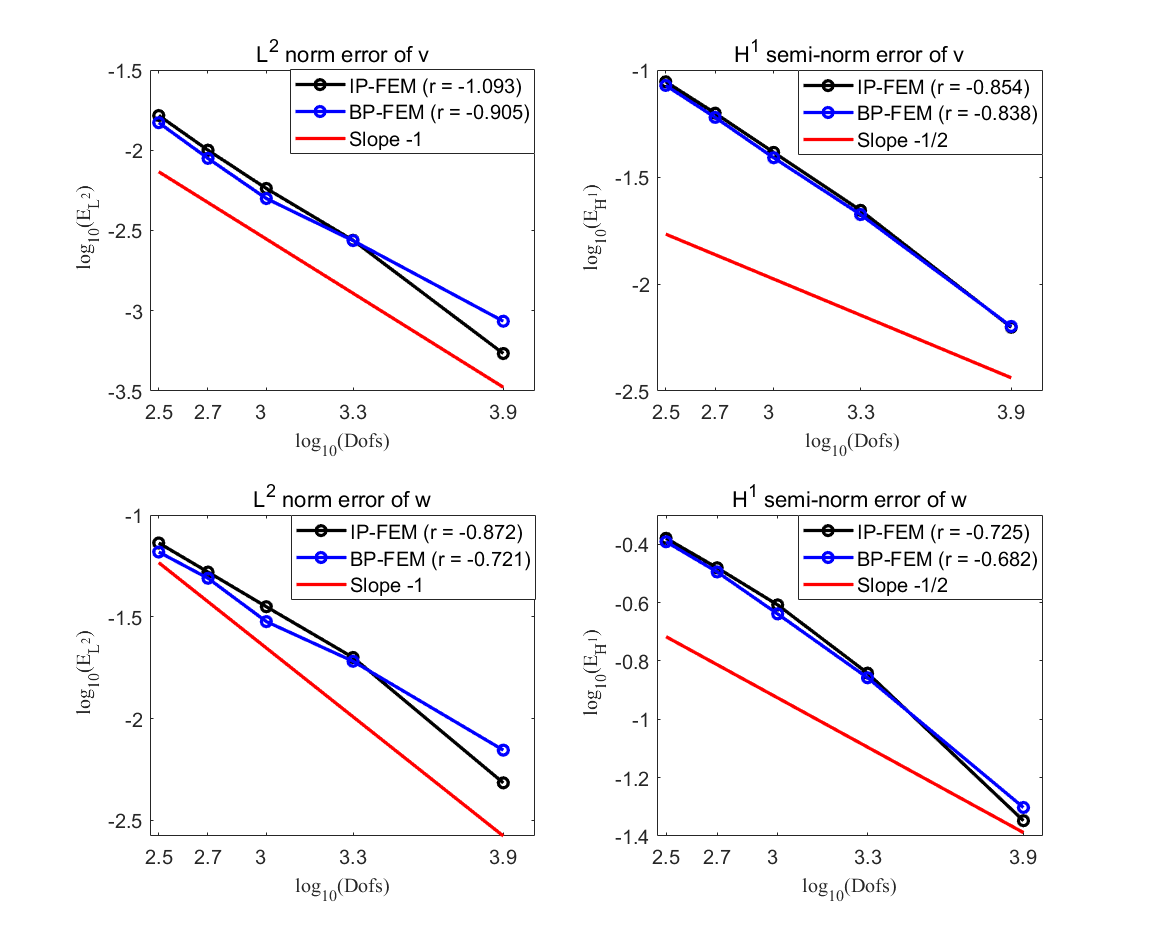}
  \caption{Example 2: The relative errors are plotted against the degrees of freedom (Dofs) for the elliptical cavity with $\gamma = \kappa \times 10^{-3}$ for the IP-FEM and with $\eta = 2.5 \kappa \times 10^{-3}$ for the BP-FEM.}
  \label{Figure_Convergence_example2}
\end{figure}

\subsection{A kite-shaped cavity}

The subsection investigates the scattering by a kite-shaped cavity. The parametric equations for the kite-shaped cavity boundary are given as follows:
\begin{equation*}
x(t) = a\cos(\theta) + b\cos(2\theta) - c, \quad y(t) = a\sin(\theta),
\end{equation*}
where the parameters are defined as $\theta \in [0, 2\pi)$, $a = 0.3$, $b = 0.2$, and $c = 0.1$. In the experiments, the open domain is truncated by a circle with a radius of $R=0.6$. Figure \ref{Fig_ex3_domain_mesh} displays the truncated domain and the mesh used for solving the kite-shaped cavity scattering problem. The remaining parameters for this problem are the same as those used in the second example. For the sake of comparison, the reference solution is obtained using the IP-FEM with $\gamma=\kappa\times 10^{-3}$ on a very fine mesh.

\begin{figure}[ht]
\centering
\includegraphics[width=0.35\textwidth]{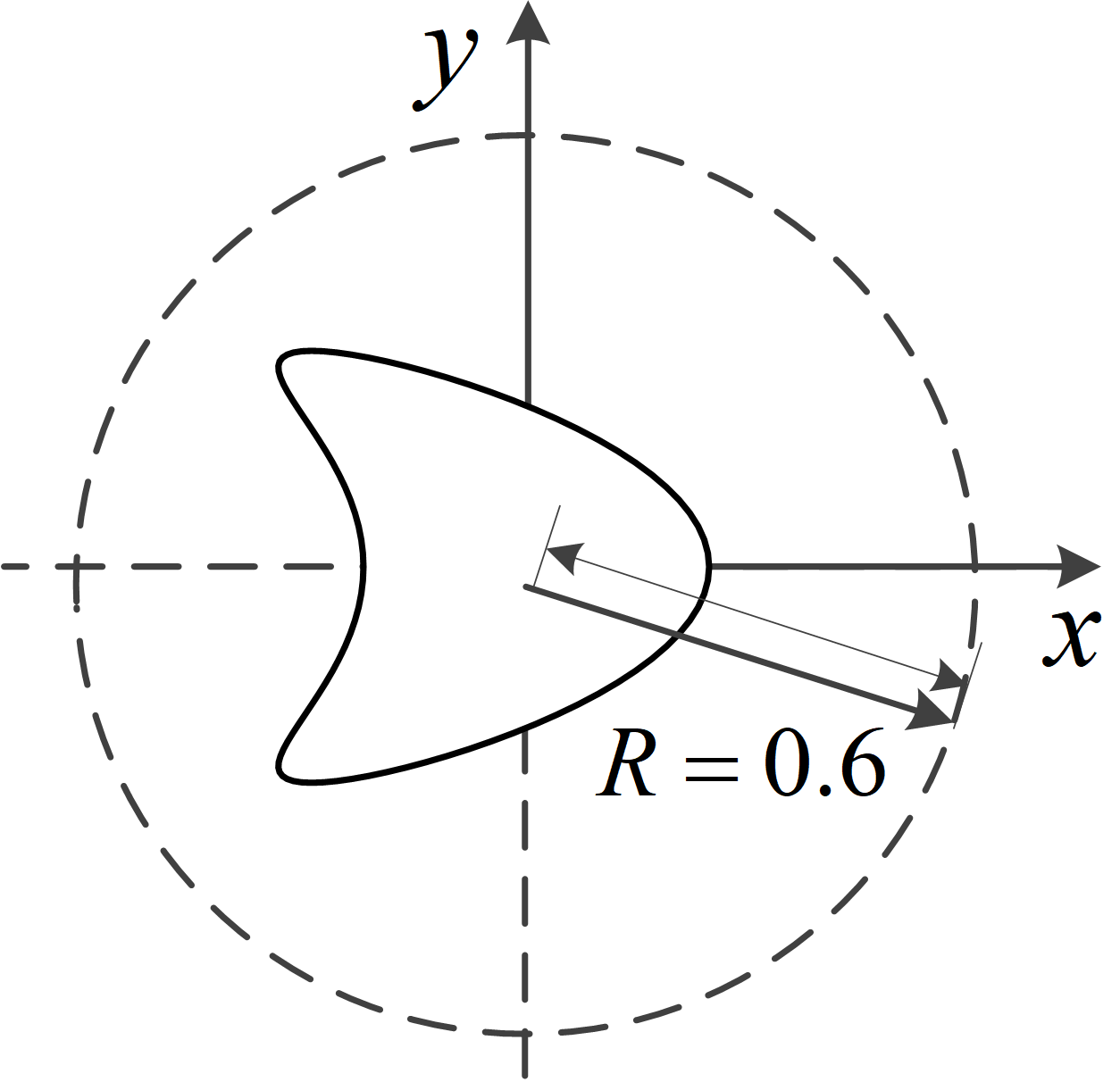}
\hspace{1cm}
\includegraphics[width=0.37\textwidth]{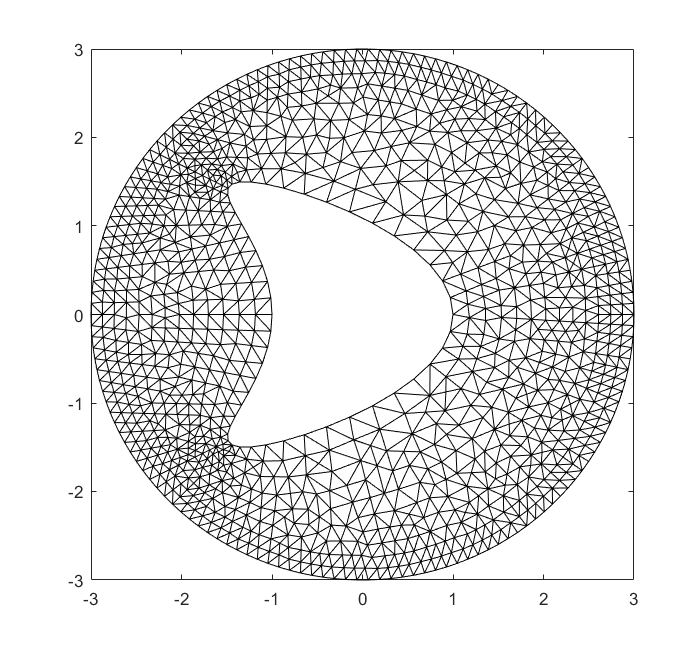}
\caption{Geometry of a kite-shaped cavity:
(a) The truncated domain; (b) A FEM mesh.}
\label{Fig_ex3_domain_mesh}
\setlength{\belowcaptionskip}{10pt}
\end{figure}

\subsubsection{Accuracy}

In this subsection, we present the regular linear FEM, the IP-FEM, and the BP-FEM to validate the effectiveness of the proposed method. The mesh size and the wavenumber are set as $h=0.05$ and $\kappa = \pi$, respectively. Figure \ref{Figure_example3_accuracy} displays the solutions $w$ obtained using the regular linear FEM ($\gamma=0$), the IP-FEM ($\gamma = \kappa\times 10^{-3}$), and the BP-FEM ($\eta = 2.5\kappa\times 10^{-3}$) on the cavity boundary. Additionally, the corresponding results of the regular linear FEM and the IP-FEM on the entire domain are also presented in the figure. It is noted that the result of the BP-FEM on the entire domain is similar to that of the IP-FEM. From these figures, we observe that both the IP-FEM and BP-FEM effectively suppress the oscillation of the bending moment on the cavity boundary compared with the regular linear FEM.

\begin{figure}
  \centering
  \includegraphics[width=0.45\textwidth]{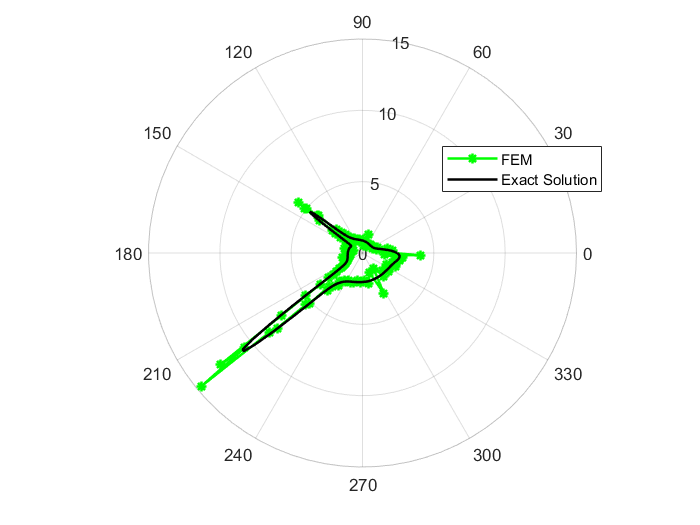}
  \includegraphics[width=0.45\textwidth]{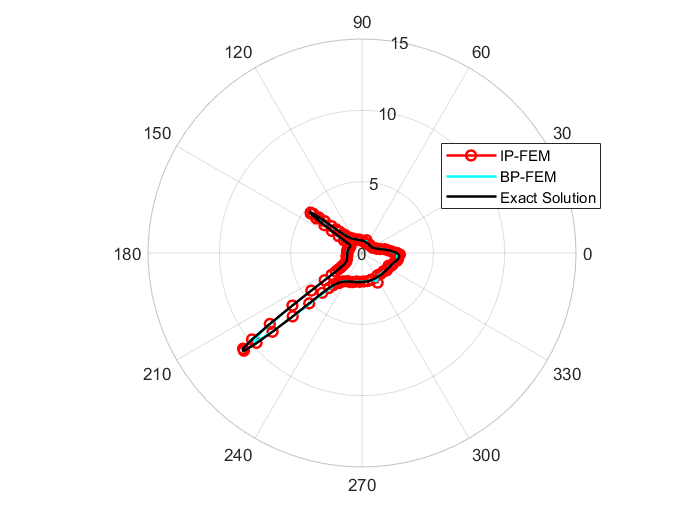}
  \includegraphics[width=0.45\textwidth]{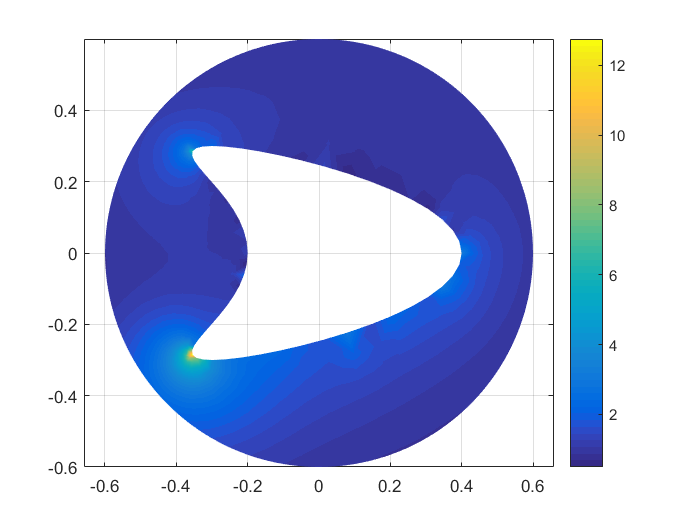}
  \includegraphics[width=0.45\textwidth]{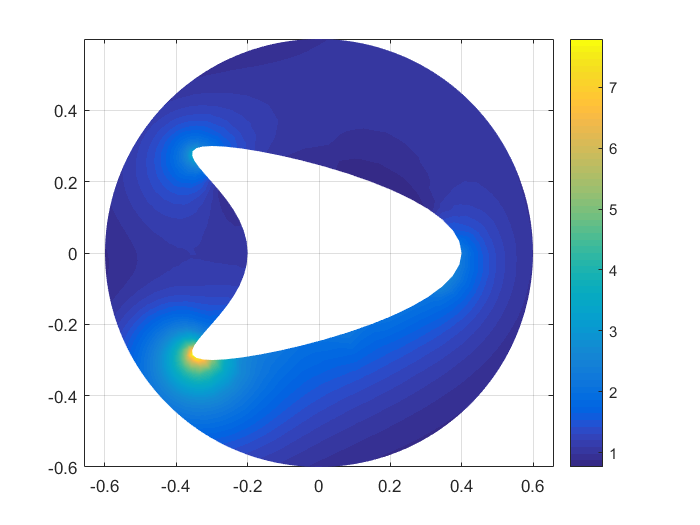}
  \caption{Example 3: The bending moment $w$ on the boundary of cavity and the entire domain:
  (left) the regular linear FEM ($\gamma=0$); (right) the IP-FEM ($\gamma = \kappa\times 10^{-3}$) and the BP-FEM ($\eta = 2.5 \kappa\times 10^{-3}$).}
  \label{Figure_example3_accuracy}
\end{figure}

\subsubsection{Convergence}

In this subsection, we investigate the convergence of the IP-FEM and the BP-FEM. Figure \ref{Figure_Convergence_example3} illustrates the convergence of the relative errors of the $L^2$ norm and the $H^1$ semi-norm for the scattered field $v$ and the bending moment $w$. From these figures, we observe that the convergence rates of $v$ and $w$ also achieve good convergence orders.

\begin{figure}
  \centering
  \includegraphics[width=0.8\textwidth]{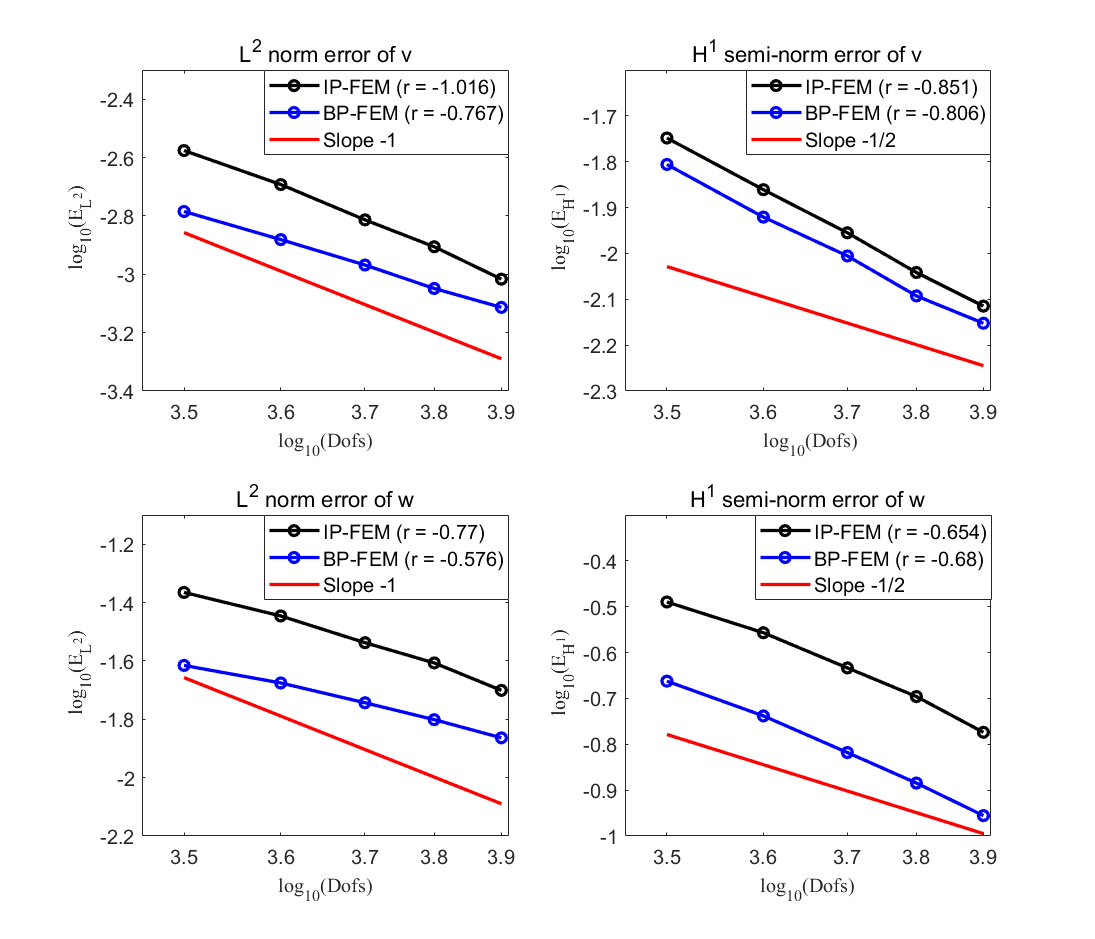}
  \caption{Example 3: The relative errors are plotted against the degrees of freedom (Dofs) for the kite-shaped cavity with $\gamma = \kappa \times 10^{-3}$ for the IP-FEM and with $\eta = 2.5\kappa \times 10^{-3}$ for the BP-FEM.}
  \label{Figure_Convergence_example3}
\end{figure}

\section{Conclusion} \label{sec7}

In this paper, we have introduced and applied the IP-FEM and the BP-FEM to investigate the flexural scattering by a clamped cavity in an infinite thin plate. The proposed model utilizes the decomposition technique and the TBC technique to transform a fourth-order problem on an unbounded domain into two second-order equations with coupled boundary conditions and TBCs on a bounded domain. To effectively suppress the oscillation of the bending moment on the cavity boundary, we have incorporated the interior penalty (IP) and boundary penalty (BP) techniques into the original variational formulation of the decomposed problem. The resulting two new variational formulations, augmented with penalty terms, are discretized using linear triangular elements.

To verify the effectiveness of the proposed method, we conducted a numerical experiment involving flexural scattering by a circle-shaped cavity, for which we obtained an analytical solution. The results of this experiment confirm that both the IP-FEM and the BP-FEM successfully suppress the oscillations of the bending moment on the cavity boundary, leading to a significant improvement in the bending moment, while maintaining the accuracy of the displacement compared to the regular linear FEM. Furthermore, we extended the model to handle flexural scattering problems with cavities of different shapes and compared the results with corresponding reference solutions. The numerical results demonstrated that the convergence rates for the displacement and bending moment achieved by the IP-FEM and the BP-FEM approach optimal convergence orders.

As part of our future work, we aim to conduct further investigations into the existence of decomposed problems using the variational method and delve deeper into the related mathematical theory of the IP-FEM and the BP-FEM and numerical calculations with complex-valued penalty parameters. These research endeavors are expected to significantly contribute to enhancing the understanding and applicability of our proposed numerical methods in the field of flexural scattering and other related problems.

\end{document}